\documentclass[11pt,a4paper]{article}
\usepackage{amsmath,amssymb,amsthm,amsfonts,latexsym,graphicx,subfigure}
\usepackage[all,import]{xy}
\usepackage[numbers,sort&compress]{natbib}
\usepackage{indentfirst}
\usepackage{fancyhdr,enumerate}
\usepackage{bbm,color}
\usepackage{tikz}
\usetikzlibrary{shapes.geometric, arrows}
\usepackage{accents,cases}
\usepackage{multirow}
\usepackage[lined, ruled, linesnumbered, longend]{algorithm2e}

\usepackage{array,float}
\newcommand{\PreserveBackslash}[1]{\let\temp=\\#1\let\\=\temp}
\newcolumntype{C}[1]{>{\PreserveBackslash\centering}p{#1}}
\newcolumntype{R}[1]{>{\PreserveBackslash\raggedleft}p{#1}}
\newcolumntype{L}[1]{>{\PreserveBackslash\raggedright}p{#1}}

\newcommand{\RQ}{\mathcal{R}}

\makeatletter
\def\wbar{\accentset{{\cc@style\underline{\mskip8mu}}}}

\makeatother

\renewcommand{\vec}[1]{\mbox{\boldmath \small $#1$}}

\newcommand{\R}{\ensuremath{\mathbb{R}}}

\theoremstyle{plain}
\newtheorem{theorem}{Theorem}

\newtheorem{defn}{Definition}[section]

\newtheorem{lemma}{Lemma}
\newtheorem{remark}{Remark}[section]

\newtheorem{cor}{Corollary}[section]
\newtheorem{pro}{Proposition}[section]
\newtheorem{example}{Example}[section]

\allowdisplaybreaks

\begin{document}

\title{Nodal domain theorems for $p$-Laplacians on signed graphs}
\author{Chuanyuan Ge\footnotemark[1]\and Shiping Liu\footnotemark[2] \and Dong Zhang\footnotemark[3]}

\footnotetext[1]{School of Mathematical Sciences, 
University of Science and Technology of China, Hefei 230026, China. \\
Email address:
{\tt gechuanyuan@mail.ustc.edu.cn} }
\footnotetext[2]{School of Mathematical Sciences, 
University of Science and Technology of China, Hefei 230026, China. \\
Email address:
{\tt spliu@ustc.edu.cn}
}
\footnotetext[3]{LMAM and School of Mathematical Sciences, 
        Peking University,  
      100871 Beijing, China.
\\
Email addresses: 
{\tt dongzhang@math.pku.edu.cn}\;\; and \; {\tt 13699289001@163.com}
}
\date{}\maketitle
\begin{abstract}
We establish various nodal domain theorems  for $p$-Laplacians on signed graphs, which unify most of the existing results on nodal domains of graph $p$-Laplacians and arbitrary symmetric matrices. Based on our nodal domain estimates, we obtain a higher order Cheeger inequality that relates the variational eigenvalues of $p$-Laplacians and Atay-Liu's multi-way Cheeger constants on signed graphs. In the particular case of $p=1$, this leads to several identities relating variational eigenvalues and multi-way Cheeger constants. Intriguingly, our approach also leads to new results on usual graphs, including a weak version of Sturm's oscillation theorem for graph $1$-Laplacians and nonexistence of eigenvalues between the largest and second largest variational eigenvalues of $p$-Laplacians with $p>1$ on connected bipartite graphs.
\end{abstract}

\section{Introduction}

Graph $p$-Laplacian is a natural discretization of the continuous $p$-Laplacian on Euclidean domains, and it is also a simple  nonlinearization of the Laplacian matrix. 
The spectrum of the graph $p$-Laplacian is closely related to many combinatorial properties of the graph itself; and its eigenpairs, reveal important information about the topology and  geometry of the graph. 
For example, similar to the original Euclidean $p$-Laplacian and graph linear Laplacian, the $p$-Laplacian on graphs has some important relations to Cheeger cut problem and shortest path problem on graphs. 
Just as the  Laplacian matrix which has been successfully used in diverse areas, the graph $p$-Laplacian has been also widely used in various applications, including spectral clustering \cite{ZS06,ZHS06,HeinBuhler2010,HeinBuhler2009}, data and image processing problems, semi-supervised learning and unsupervised learning \cite{UJT21,ZS06,ZHS06}. 
Much recent work has shown that algorithms based on the graph $p$-Laplacian perform better than classical algorithms based on the linear Laplacian in solving these practical problems in image science. 

The theoretical aspects of $p$-Laplacians on graphs and networks are still not well understood due to the nonlinearity. 
Among several  progresses in this direction, a remarkable development is that the second eigenvalue has a mountain-pass characterization and  it is a variational eigenvalue which satisfies the Cheeger inequality \cite{Amghibech,HeinBuhler2009}. 
 Another important result is the nodal domain count for graph $p$-Laplacians, including an interesting  relation that connects the nodal domains of the $p$-Laplacian  and the multi-way Cheeger constants on graphs \cite{TudiscoHein18}. 
 For the limiting case $p=1$, the spectral theory for graph 1-Laplacian was proposed by Hein and B\"uhler \cite{HeinBuhler2010} for 1-spectral clustering, and was latter studied by Chang  \cite{Chang}   from a variational point of view. 
 For example, Cheeger’s constant, which has only some upper and lower bounds given by the second eigenvalues of  $p$-Laplacians with $p>1$,  equals the second  eigenvalue of graph 1-Laplacian \cite{HeinBuhler2010,Chang}. 
 Moreover, any Cheeger set can be identified with any strong nodal domain of any eigenfunction corresponding to the second eigenvalue of graph 1-Laplacian.
 
 To some extent, nodal domain theory provides a good perspective for understanding the spectrum of graph $p$-Laplacians. 
Indeed, various versions of discrete nodal domain theory have been developed in different contexts. 
A very useful context  should be the signed graphs, whose spectral theory has led to a number of breakthroughs in theoretical computer science and combinatorial geometry,  including the solutions to the  sensitivity conjecture \cite{Huang19} and the open problems on equiangular lines \cite{CKLY22,JTYZZ21,JTYZZ}. In addition, signed  graphs have many other practical applications on modeling  biological networks, social situations, ferromagnetism, and general signed networks \cite{Harary53, ArefWilson19,ArefMasonWilson20}. 
Therefore, it should be natural and useful to develop a general spectral theory that includes  nodal domain theorems on signed graphs. 
Along this line,  Ge and Liu  \cite{GL21+} provided  a definition of the strong and weak nodal domains on signed graphs, which is compatible with the classical one in \cite{DGLS01} on graphs. 
They also obtained  sharp estimates of the number of strong and weak nodal domains for generalized linear Laplacian on signed graphs. We notice that estimates of strong nodal domains on signed graphs has been established in an earlier work of Mohammadian \cite{Ali16}, see \cite[Remark 3.12]{GL21+}.
For more details and historical background of nodal domain theory, we refer the readers to \cite{GL21+}. We particularly mention that the results in Fiedler's classical 1975 paper \cite{Fiedler75} can be considered as nodal domain theorems on signed trees (see \cite[Section 5]{GL21+}). In 2013, Berkolaiko \cite{Berkolaiko13} and Colin de Verdi\`{e}re \cite{CdV13} computed the nodal count of edges on signed graphs by allowing the signs of each edge to become complex. See  Remark \ref{remark:2.2} for more detailed comments.

The combination of signed versions and nonlinear analogs of nodal domain theorems is the main focus of this paper. 
To the best of our knowledge, the $p$-Laplacian on signed graphs has not been well studied. 
A related research was  given in \cite{JMZ} for $p$-Laplacians on  oriented hypergraphs,  which includes the $p$-Laplacian on  signed graphs as a special case. 
However, that paper does not focus on the nodal domain property, so  there are no sufficiently in-depth results on nodal domain theorems for $p$-Laplacians on signed graphs. 

In this paper, we systematically establish a nodal domain theory for $p$-Laplacians on signed graphs, which unifies the ideas and approaches from these recent works \cite{TudiscoHein18,DPT21,GL21+,JMZ}. 
Based on our nodal domain estimates, we also obtain a  higher order Cheeger inequality  that relates the variational eigenvalues of $p$-Laplacians and Atay-Liu's multi-way Cheeger constants on signed graphs \cite{AtayLiu}. 
Although these results appear to be formally similar to that in \cite{TudiscoHein18,DPT21}, there are several key differences in both results and approaches. 
First, our upper bounds for the number of dual nodal domains for  $p$-Laplacians on signed graphs are new, and the proof relies heavily on the intersection property of Krasnoselskii genus. 
In particular, for $p>1$, the estimate of the number of dual weak  nodal domains, and the bound on the number of dual strong  nodal domains of the $k$-th eigenfunction with  minimal support, further require the odd homeomorphism deformation lemma in Struwe's book \cite{Struwe}; while the case of $p=1$ should be treated separately by using the localization property. 
It is worth noting that a cautious analysis gives us a stronger result for the signed 1-Laplacian case, which is also new for graph 1-Laplacian. 
Second, the approach we use to obtain the lower bound estimates for the number of strong nodal domains, further  relies on a duality argument by considering the quantity $\mathfrak{S}(f)+\overline{\mathfrak{S}}(f)$, which is similar to the linear case in \cite{GL21+}, but the nonlinear estimate requires more subtle techniques. 
Third, the $k$-way Cheeger inequality  connecting variational eigenvalues of $p$-Laplacians and Atay-Liu's $k$-way Cheeger constants on signed graphs is essentially new, although the proof is not difficult to anyone who is familiar with analysis or   spectral graph theory. 
Interestingly, this result also reveals that variational eigenvalues of the $1$-Laplacian on signed graphs are very closely related to certain combinatorial quantities on signed graphs. 
Fourth, it should be noted that many of the nodal domain properties of $p$-Laplacians are different on graphs of different signatures. 
For example, on a balanced graph,  the second eigenfunction has exactly two weak nodal domains (see \cite{DPT21}), which is not always the case on an unbalanced graph, see Example \ref{ex:weak}.
Very interestingly, we prove a nonlinear Perron-Frobenius theorem for $p$-Laplacians on antibalanced graphs, that is, the eigenfunction corresponding to the largest eigenvalue is  positive everywhere or negative everywhere. Moreover, the eigenfunction corresponding to the largest eigenvalue is unique up to a constant multiplication. 
 However, this does not hold for $p$-Laplacians on balanced graphs.

Even on the usual graphs, our theorems directly derive at least two new results:
\begin{itemize}
\item Any eigenfunction corresponding to the $k$-th variational eigenvalue $\lambda_k$ (such that $\lambda_k>\lambda_{k-1}$) of the graph 1-Laplacian with minimal support has at least $k+r-2$ zeros, where $r$ is the variational multiplicity of $\lambda_k$ (see Theorem \ref{thm:min nodal domain}). 
Recall Sturm’s oscillation theorem, which says that the $k$-th eigenfunction of the second-order linear ODE has exactly $(k-1)$ zeros. Our result actually shows that the $k$-th variational  eigenfunction of the graph 1-Laplacian with minimal support has at least $(k-1)$ zeros. Therefore, in a sense, we are actually building a weak version of Sturm's theorem for the graph 1-Laplacian. 
\item When $p>1$, there are no other eigenvalues between the largest and the second largest variational eigenvalues of the graph $p$-Laplacian on connected bipartite graphs (see Corollary \ref{cor:5.1}). 
This new phenomenon can be seen as a dual version of the classic result that there are no other eigenvalues between the smallest and the second smallest  variational  eigenvalues of the graph $p$-Laplacian.
\end{itemize}

The paper is structured as follows. 
In Section \ref{sec:pre}, we  collect preliminaries on $p$-Laplacians and signed graphs, particularly on the continuity and switching property of $p$-Laplacian spectrum of signed graphs. 
In Section \ref{sec:upper-nodal}, we present the upper bounds of strong and weak nodal domains for $p$-Laplacians on signed graphs, and discuss the related nodal domain properties on forests. 
In Section \ref{sec:Cheeger}, we
show multi-way Cheeger inequalities related to strong nodal domains involving $p$-Laplacians on  signed graphs. 
In Section \ref{sec:Perron}, we
establish a nonlinear   Perron-Frobenius theorem for the largest eigenvalue of the $p$-Laplacian on antibalanced graphs. 
In Section \ref{sec:interlacing}, we develop the interlacing theorem which is a signed version of Weyl-like inequalities proposed in \cite{DPT21}. 
Finally, we show lower bound estimates for the number of strong nodal domains in 
Section \ref{sec:lower-nodal}.

\section{Preliminaries}
\label{sec:pre}
To explain the interesting  story clearly, let us present our setting and notations in this section.    

Let $G=(V,E)$ be a finite graph with a positive edge measure $w:E\to \R^+$, a vertex weight $\mu:V=\{1,2\cdots,n\}\to\R^+$ and a real potential function  $\kappa:V\to\R$. In this paper, we work on a  signed graph $\Gamma=(G,\sigma)$ with an additional  signature $\sigma:E\to \{-1,1\}$. We use $C(V)$ to denote the set of all the real functions on $V$, and we always identify $C(V)$ with $\R^n$, i.e.,  $C(V)\cong \R^n$. We  denote $w(\{x,y\})$, $\kappa(x)$, $\mu(x)$ and $\sigma(\{{xy}\})$ by $w_{xy}$,  $\kappa_x$,  $\mu_x$ and $\sigma_{xy}$ for simplicity.

In this paper, we assume $p\geq 1$. Let $\Phi_p:\R\to \R$ be  defined as $\Phi_p(t)=|t|^{p-2}t$ if $t\neq 0$ and  $\Phi_p(t)=0$ if $t=0$. We also write $x\sim y$ when $\{x,y\}\in E$.

For $p>1$, the signed  $p$-Laplacian  $\Delta_p^\sigma:C(V)\to C(V)$ is defined \cite{Amghibech,DPT21} by
$$\Delta_p^\sigma f(x)=\sum_{y\sim x}w_{xy}\Phi_p(f(x)-\sigma_{xy}f(y))+\kappa_x\Phi_p(f(x)),\;\; x\in V,\;f\in C(V).$$

A nonzero function $f:V\to \R$ is an eigenfunction of $\Delta_p^{\sigma}$ associated with the eigenvalue $\lambda$ if the following identity holds
$$\Delta_p^\sigma f(x)=\lambda\mu_x\Phi_p(f(x)),\;\;\forall x\in V.$$

The signed  $1$-Laplacian $\Delta_1^\sigma$ \cite{Chang, CSZ16+} is a set-valued map  defined by
$$\Delta_1^\sigma f(x)=\left\{\left.\sum_{y\sim x}w_{xy}z_{xy}+\kappa_xz_x\right|z_{xy}\in \mathrm{Sgn}(f(x)-\sigma_{xy}f(y)),z_{xy}=-\sigma_{xy}z_{yx},z_x\in\mathrm{Sgn}(f(x)) \right\},$$
in which $$\mathrm{Sgn}(t):=\begin{cases}
 \{1\} & \text{if } t>0,\\
 [-1,1] & \text{if }t=0,\\
 \{-1\} & \text{if }t<0.
 \end{cases}
$$
In this paper, we always use $\mathrm{Sgn}$ to denote the above set-valued sign function. And we use $\mathrm{sgn}$ to denote the usual sign function as follows
$$\mathrm{sgn}(t):=\begin{cases}
 1 & \text{if } t>0,\\
 0 & \text{if }t=0,\\
 1\ & \text{if }t<0.
 \end{cases}
$$

For a nonzero function $f:V\to \mathbb{R}$, we say that it is an eigenfunction of $\Delta_1^{\sigma}$ corresponding to an eigenvalue $\lambda\in \R$ if $$\Delta_1^\sigma f(x)\bigcap \lambda \mu_x \mathrm{Sgn}(f(x))\ne\emptyset,\;\;\forall x \in V,$$ or equivalently, the differential inclusion
$$0\in \Delta_1^\sigma f(x)- \lambda \mu_x \mathrm{Sgn}(f(x)),\;\;\forall x \in V$$
holds in the language of Minkowski sum of convex sets. 

We will also discuss eigenfunctions with \emph{minimal supports} (see Theorem \ref{thm:min nodal domain} in the next section).  
 \begin{defn}
     For any function $g:V\to \mathbb{R}$, define $\mathrm{supp}(g):=\{x\in V:g(x)\neq 0\}$. Let $f$ be an eigenfunction of $\Delta_p^{\sigma}$ corresponding to $\lambda$. We say $f$ has minimal support if for any eigenfunction $g$ of $\Delta_p^{\sigma}$ corresponding to $\lambda$ with $\mathrm{supp}(g)\subset\mathrm{supp}(f)$, we must have $\mathrm{supp}(g)=\mathrm{supp}(f)$.
 \end{defn}

\begin{defn}[Switching]
   A function $\tau$ is called a switching function if it maps from $V$ to $\{+1,-1\}$. Switching the signature of $\Gamma=(G,\sigma)$ by $\tau$ refers to the operation of changing $\sigma$ to be $\sigma^\tau $ where
	\[\sigma^{\tau}_{xy}:=\tau(x)\sigma_{xy}\tau(y)\]
	for any $\{x,y\}\in E$.
\end{defn}

\begin{defn}
	 Two signed graphs $\Gamma=(G,\sigma)$ and $\Gamma'=(G.\sigma')$ are called to be switching equivalent if there exists a switching function $\tau$ such that $\sigma'=\sigma^\tau$.
\end{defn}
Next, we define balanced and antibalanced graphs. The definition given below is equivalent to the original one by Harary \cite{Harary53} due to Zaslavsky's switching lemma
\cite{Zaslavsky82}.

\begin{defn}
    A balanced (resp., antibalanced) graph is a signed graph which is switching equivalent to a graph whose edges are all positive (resp., negative). 
\end{defn}

\begin{remark}
For $\kappa=0$, $\Delta_p^\sigma$ is the usual $p$-Laplacian on  signed graphs.

For $\sigma\equiv +1$, $\Delta_p^\sigma$ is nothing but the usual $p$-Schr\"odinger operator on  graphs. It is known that the graph $p$-Schr\"odinger eigenvalue problem covers the Dirichlet $p$-Laplacian eigenvalue problem on graphs, see, e.g., \cite{HW20}. 
 
For $p=2$, $\Delta_2^\sigma$ reduces to an arbitrary symmetric matrix by  taking certain parameters $w$, $\sigma$, $\mu$ and $\kappa$. 
 \end{remark}
Before giving the following definition, we recall that a set $S$ in a Banach space is centrally symmetric if $S=-S$ where $-S:=\{-x:\,x\in S\}$. 
\begin{defn}[index]
The \emph{index} (or Krasnoselskii 
genus) of a compact  centrally symmetric set $S$ in a Banach space  is defined by 
\begin{equation*}
\gamma(S):=\begin{cases}
0 &\text{if}\;S=\emptyset \\
\min\limits\{k\in\mathbb{Z}^+: \exists\; \text{odd continuous}\; h: S
\to \mathbb{R}^{k}\} & \text{if}\; S\neq\emptyset.
\end{cases}
\end{equation*}
If, in the above, $\{k\in\mathbb{Z}^+:\exists\; \text{odd continuous}\; h: S\to \mathbb{R}^{k}\}=\emptyset$,  we set \[\min\limits\{k\in\mathbb{Z}^+: \exists\; \text{odd continuous}\; h: S\to \mathbb{R}^{k}\}=\infty.\]
\end{defn}
The following proposition can be found in \cite[Proposition 5.2]{Struwe}.
\begin{pro}\label{index}
    For any bounded centrally symmetric neighborhood $\Omega$ of the origin in $\mathbb{R}^m$, we have $\gamma(\partial \Omega)=m$.
\end{pro} 

Let $\mathcal{S}_p(V)=\{f\in C(V):\sum_{x\in V}\mu_x|f(x)|^{p}=1\}$, and let $$\mathcal{F}_{k}\left(\mathcal{S}_p(V)\right)=\{A\subset \mathcal{S}_p(V):A\text{ is compact  centrally symmetric and }\gamma(A)\ge k\}.$$ For convenience, we omit the symbol $V$ if no confusion arises, e.g. $S_p:=S_p(V),\,\mathcal{F}_{k}(S_p):=\mathcal{F}_{k}(S_p(V))$. Denote by  $$\RQ_p^\sigma(f)=\frac{\sum_{\{x,y\}\in E}w_{xy}|f(x)-\sigma_{xy}f(y)|^{p}+\kappa_x|f(x)|^p}{\sum_{x\in V}\mu_x|f(x)|^p}$$
the $p$-Rayleigh quotient. 
The Lusternik-Schnirelman theory
allows us to define a sequence of variational  eigenvalues  of $\Delta_p^\sigma$:
$$ \lambda_k(\Delta_p^\sigma):=\inf_{S\in \mathcal{F}_{k}( \mathcal{S}_p)}\sup\limits_{f\in S}  \RQ_p^\sigma(f),\;\;k=1,2,\ldots,|V|.$$
Moreover, each variational eigenvalue is an eigenvalue of $\Delta_p^{\sigma}$.

It is worth noting that there does exist graphs with non-variational eigenvalues, see \cite[Theorem 6]{Amghibech}. It is proved in \cite[Theorem 3.7]{DPT21} that forests admit
only variational eigenvalues. 
\begin{defn}[eigenspace]
The {\sl eigenspace}   $\mathsf{X}_{\lambda}(\Delta_p^\sigma)$ of $\Delta_p^{\sigma}$ corresponding to an eigenvalue $\lambda$ is the subset of $\mathcal{S}_p$ consists of the all eigenfunctions corresponding to  $\lambda$. 
\end{defn}

The {\sl multiplicity} of an eigenvalue $\lambda$ of $\Delta_p^\sigma$ is defined to be $\gamma(\mathsf{X}_{\lambda}(\Delta_p^\sigma))$, and we shall denote it by $\mathrm{multi}(\lambda(\Delta_p^\sigma))$.
In this paper, we write $\lambda_k$ to denote $\lambda_k(\Delta_p^{\sigma})$, if it is clear.

\begin{defn}[variational multiplicity]
For a variational eigenvalue $\lambda$ of $\Delta_p^\sigma$, its {\sl variational multiplicity} is defined as the number of times $\lambda$ appears in the sequence of variational eigenvalues. We will  denote it by $\mathrm{multi}_v(\lambda(\Delta_p^\sigma))$.
\end{defn}
It is known that for any variational eigenvalue, its variational multiplicity is always less than or equal to its multiplicity \cite[Lemma 5.6]{Struwe}.  

\begin{defn}[nodal domains, Definitions 3.1--3.4 in \cite{GL21+}]\label{def:nodal}
	Let $\Gamma=(G,\sigma)$ be a signed graph  and $f: V\to \mathbb{R}$ be a function. A sequence $\{x_i\}_{i=1}^k$  of vertices is called a \emph{strong nodal domain walk} of $f$ if $x_i\sim x_{i+1}$ and $f(x_i)\sigma_{x_ix_{i+1}} f(x_{i+1})>0$ for each $i=1,2,\ldots,k-1$.
	
A sequence $\{x_i\}_{i=1}^k$, $k\geq 2$ of vertices is called a \emph{weak nodal domain walk} of $f$ if for any two consecutive non-zeros $x_i$ and $x_j$ of $f$, i.e., $f(x_i)\neq 0$, $f(x_j)\neq 0$, and $f(x_{\ell})=0$ for any $i<\ell<j$, it holds that
\[f(x_{i}) \sigma_{x_{i}x_{i+1}}\cdots \sigma_{x_{j-1}x_{j}}f(x_j)>0.\]
We remark that every walk containing at most 1 non-zeros of f is a weak nodal domain walk.

Let $\Omega=\{ x\in V:f(x)\neq 0 \}$ be the set of non-zeros of $f$.
\begin{itemize}
  \item [(i)] Define an equivalence  relation $\mathop{\sim}\limits^S$ on $\Omega$ as follows: For any $x,y\in \Omega$, $x  \mathop{\sim}\limits^S y$ if and only if $x=y$ or there exists a strong nodal domain walk connecting $x$ and $y$. 
  
  We denote by $\{S_i\}_{i=1}^{n_S}$ the equivalence classes of the relation $\mathop{\sim}\limits^S$ on $\Omega$. We call the induced subgraph of each $S_i$ a \emph{strong nodal domain} of the function $f$.
 We denote the number $n_S$ of strong nodal domains of $f$ by $\mathfrak{S}(f)$.

  \item [(ii)] Define an  equivalence  relation $\mathop{\sim}\limits^W$ on $\Omega$ as follows: For any $x,y\in \Omega$, $x\mathop{\sim}\limits^W y$ if and only if $x=y$ or there exists a weak nodal domain walk connecting $x$ and $y$.
  
  We denote by $\{W_i\}_{i=1}^{n_W}$ the equivalence classes of the relation $\mathop{\sim}\limits^W$ on $\Omega$. We call the induced subgraph of each set \[W_i^0:= W_i\cup\{v\in V: \text{there exists  a weak nodal domain walk from $v$ to some vertex in $W_i$}\}\] a \emph{weak nodal domain} of the function $f$. We denote the number $n_W$ of weak nodal domains of $f$ by $\mathfrak{W}(f)$.
\end{itemize}
\end{defn}
Note that $\{W_i\}_{i=1}^{n_W}$ is a partition of $\Omega:=\{x\in V:f(x)\neq 0\}$. And $W_i^0$ is obtained by adding some zeros to $W_i$.

Next, we give two examples to illustrate this definition.
\begin{example}\label{ex:triangle}
	We consider the signed graph $\Gamma=(G,\sigma)$ which is shown in Figure \ref{fig:product_scheme}. The corresponding signed Laplacian is given as below
 \[\Delta_2^\sigma=\left(
\begin{array}{cccccc}
	 1 & 0& 1 & 0 & 0    \\
      0 & 1 &1 &  0 & 0   \\
      1 &  1 & 4 &  1 & 1  \\
      0 &  0 & 1 &  1 & 1  \\
	 0 &  0 & 1 &  1 & 1  \\
\end{array}
\right).\]

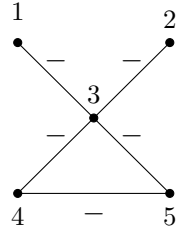
\begin{figure}[!htp]
	\centering
	\tikzset{vertex/.style={circle, draw, fill=black!20, inner sep=0pt, minimum width=3pt}}
	\begin{tikzpicture}[scale=1.0]
 	 \draw (-1,1) -- (0,0) node[midway, above, black]{$-$}
		-- (1,1) node[midway, above, black]{$-$};
		\draw (0,0) -- (-1,-1) node[midway,above, black]{$-$}
		-- (1,-1) node[midway,below, black]{$-$}
		-- (0,0) node[midway, above, black]{$-$} ;

		\node at (0,0) [vertex, label={[label distance=0mm]90: \small $3$}, fill=black] {};
           \node at (-1,-1) [vertex, label={[label distance=0mm]270: \small $4$} ,fill=black] {};
		\node at (1,-1) [vertex, label={[label distance=0mm]270: \small $5$} ,fill=black] {};
		\node at (1,1) [vertex, label={[label distance=0mm]90: \small $2$} ,fill=black] {};
	\node at (-1,1) [vertex, label={[label distance=1mm]90: \small $1$} ,fill=black] {};
		
	\end{tikzpicture}
	\caption{$\Gamma=(G,\sigma)$.}
	
	\label{fig:product_scheme}
\end{figure}
By numerical computation, we have the eigenvalues of $\Delta_2^\sigma$ 
\[\lambda_{1}=0 \leq \lambda_{2}\approx0.238\leq \lambda_{3}=1\leq \lambda_{4}\approx1.637\leq\lambda_{5}\approx5.125,\]
and the corresponding 
eigenfunctions
\begin{align*}
f_1&=(0,0,0,-1,1)^T,\\
f_2&\approx(2.313,2.313,-1.762,1,1)^T,\\
f_3&=(-1,1,0,0,0,0)^T,\\
f_4&\approx(-0.517,-0.517,-0.363,1,1)^T,\\
f_5&\approx(0.758,0.758,3.125,1,1)^T.
\end{align*}
In Table \ref{t}, we list the strong and weak nodal domains of each eigenfunction. Notice that we only provide vertex subsets. The strong and weak nodal domains are the induced subgraphs of those vertex subsets.
\begin{table}[!ht]
	\begin{center}
		\begin{tabular}{c|c|c} 
			\textbf{Eigenfunction} & \textbf{Strong nadal domain} & \textbf{Weak nodal domain}\\
		
			\hline
			$f_1$ & $\{4,5\}$& $\{1,2,3,4,5\}$\\
			$f_2$ & $\{1,2,3,4,5\}$& $\{1,2,3,4,5\}$\\
			$f_3$ & $\{1\},\{2\}$ & $\{1,2,3,4,5\}$\\
			$f_4$&$\{1\}\{2\},\{3,4,5\}$&$\{1\}\{2\},\{3,4,5\}$\\
			$f_5$&$\{1\},\{2\},\{3\},\{4\},\{5\}$&$\{1\},\{2\},\{3\},\{4\},\{5\}$\\
		\end{tabular}
\caption{Strong and weak nodal domains}\label{t}
	\end{center}
\end{table}

It is worth noting that for the eigenfunction $f_3$, vertices $1$ and $2$ lie in the same weak nodal domain because $1\to 3\to 4\to 5\to 3\to 2$ is a weak nodal domain walk.


		
\end{example}
\begin{example}
     We consider a signed star graph $\Gamma=(G,\sigma)$ depicted in Figure \ref{fig:star} and its signed Laplacian matrix: 
     \[\Delta_2^\sigma=\left(
\begin{array}{cccccc}
	 4 & 1& 1 & 1 & 1    \\
      1 & 1 &0 &  0 & 0   \\
      1 &  0 & 1 &  0 & 0  \\
      1 &  0 & 0 &  1 & 0  \\
	 1 &  0 & 0 &  0 & 1  \\
\end{array}
\right).\]
\begin{figure}[!htp]
	\centering
	\tikzset{vertex/.style={circle, draw, fill=black!20, inner sep=0pt, minimum width=3pt}}
	\begin{tikzpicture}[scale=1.0]
 	 \draw (-1.5,0) -- (0,0) node[midway, above, black]{$-$}
		-- (0,1.5) node[midway, right, black]{$-$};
		\draw (1.5,0) -- (0,0) node[midway,above, black]{$-$}
		-- (0,-1.5) node[midway, right, black]{$-$} ;

		\node at (0,0) [vertex, label={[label distance=0mm]45: \small $1$}, fill=black] {};
           \node at (-1.5,0) [vertex, label={[label distance=0mm]270: \small $4$} ,fill=black] {};
		\node at (1.5,0) [vertex, label={[label distance=0mm]270: \small $2$} ,fill=black] {};
		\node at (0,1.5) [vertex, label={[label distance=0mm]90: \small $3$} ,fill=black] {};
	\node at (0,-1.5) [vertex, label={[label distance=1mm]0: \small $5$} ,fill=black] {};
		
	\end{tikzpicture}
	\caption{The signed star graph.}\label{fig:2}	
	\label{fig:star}
\end{figure}
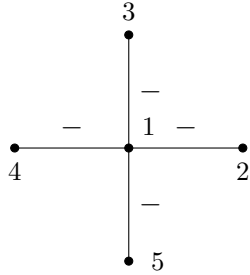

The eigenvalues of $M$ are $\lambda_1=0<\lambda_2=\lambda_3=\lambda_4=1<\lambda_5=5$. We consider the eigenfunction $f=(0,1,1,-1,-1)$ corresponding to $\lambda_2$. It is direct to check that there are $4$ strong nodal domains of $f$. Next, we investigate the weak nodal domains. Observe that $3\to 1\to 2$ and $4\to 1\to 5$ are both weak nodal domain walks of  $f$. And there are no weak nodal domain walks between $\{2,3\}$ and $\{4,5\}$. Using the notation of Definition \ref{def:nodal}, we have $W_1=\{2,3\}$ and $W_2=\{4,5\}$. Furthermore, we have $W_1^0=\{1,2,3\}$ and $W_2^0=\{1,4,5\}$. That is, $f$ has two weak nodal domains.

		
\end{example}
Now, we recall two propositions from \cite[Propostions 3.16 and 3.17]{GL21+} which will be useful later in the proof of Theorem \ref{thm:weak nodal domain}. 
\begin{pro}\label{pro:connected}
Let $\{D_i\}_{i=1}^q$ be the all weak nodal domains of a non-zero function $f$ on a
signed graph $\Gamma=(G,\sigma)$. Let $G_D = (V_D, E_D)$ be the graph given by
$$V_D:=\{D_i\}_{i=1}^q,\,\text{and\,}\,E_D:=\big\{\{{D_i,D_j\}:D_i\sim D_j}\big\},$$
where $D_i\sim D_j$ means that there exist $x\in D_i$ and $y\in D_j$ such that $x\sim y$. Then, if the graph $G$ is connected, so does the graph $G_D$.
\end{pro}
\begin{pro}\label{pro:cap}
    Let $f$ be a non-zero function on a signed graph $\Gamma=(G,\sigma)$. Then for any three weak nodal domain $D_1,D_2,D_3$ of $f$, we have $D_1\cap D_2\cap D_3=\emptyset$.
\end{pro}
\begin{remark}\label{remark:2.2}
    Another way to study the discrete nodal domains is to consider the edges instead of vertices. Given a function $f$,  define two edge sets $E^+=\{\{x,y\}\in E:f(x)\sigma_{xy}f(y)>0 \}$, $E^-=\{\{x,y\}\in E:f(x)\sigma_{xy}f(y)<0 \}$ and a vertex set $V_0=\{x\in V:f(x)\neq 0\}$. Then the number of strong nodal domains of a function $f$ is equal to the number of connected components of the graph $\Gamma'=(G',\sigma')$ where $G'=(V_0,E^+)$ and $\sigma'_{xy}=\sigma_{xy}$ for any $\{x,y\}\in E^+$. Mohammadian \cite{Ali16} proved the upper bound of the signed strong nodal domains by considering the graph $\Gamma'$. When $f$ is a generic eigenfunction, i.e., $f$ is simple and  non-zero on every vertex, the set $E^-$ is regarded as the nodal set of $f$, and the cardinality of $E^-$ is called the nodal count of $f$. The properties of nodal count have been studied in, e.g., \cite{Berkolaiko08,Berkolaiko13,CdV13,AG22}. The nodal count of signed Laplacian plays an important role in the extension of the Nodal Universality Conjecture from quantum graphs \cite{ABB22} to discrete graphs \cite{AG22}.
\end{remark}
We use $\overline{\mathfrak{S}}(f)$ (resp., $\overline{\mathfrak{W}}(f)$)  to denote the number of strong (resp., weak) nodal domains of $f$ with respect to $(G,-\sigma)$. 

The perturbation theory plays an important role in  studying of the properties of linear operators \cite{Kato76}. The following proposition is about the 
 perturbation theory of eigenvalues of $p$-Laplacian. To state the proposition, we first recall the definition of upper hemi-continuity of set-valued maps.
 \begin{defn}
     Let $X$ and $Y$ be metric spaces. A set-valued map $F:X\to \mathcal{P}(Y)$, where $\mathcal{P}(Y)$ stands for the collection of all subsets of $Y$, is called upper hemi-continuous at $x\in X$ if for any neighborhood $U$ of $F(x)$ in $Y$, there exists $\eta >0$, such that for any $x'\in B_X(x,\eta):=\{z\in X:d_X(x,z)<\eta\}$ where $d_X$ is the metric of $X$, we have $F(x')\subset U$. It is said to be upper hemi-continuous if it is upper hemi-continuous at any point of $X$.
 \end{defn}

 \begin{pro}\label{pro:continuity}
 The $k$-th variational eigenvalue is continuous  with respect to   $$(w,\kappa,\mu)\in(0,+\infty)^E\times\R^V\times (0,+\infty)^V.$$ Moreover, the multiplicity and variational multiplicity of the $k$-th variational eigenvalue are both upper semi-continuous with respect to $(w,\kappa,\mu)$  and the corresponding eigenspace is upper hemi-continuous with respect to $(w,\kappa,\mu)$.  In particular, the set of the parameters $(w,\kappa,\mu)$ such that $\lambda_k(\Delta_p^\sigma)$ has multiplicity $1$ is open in $(0,+\infty)^E\times\R^V\times (0,+\infty)^V$. Similarly, the set of the parameters $(w,\kappa,\mu)$ such that the variational multiplicity of $\lambda_k(\Delta_p^\sigma)$ is $1$ is also open in $(0,+\infty)^E\times\R^V\times (0,+\infty)^V$.
 \end{pro}

 \begin{proof}
  Since $\mathcal{R}_p^\sigma$ is locally Lipschitz continuous  with respect to   \[(w,\kappa,\mu,f)\in(0,+\infty)^E\times\R^V\times (0,+\infty)^V\times(\R^V\setminus\{\vec0\}),\] it is easy to show that 
the $k$-th variational eigenvalue $$\lambda_k=\inf_{S\in\mathcal{F}_{k}(\mathcal{S}_p)}\sup\limits_{f\in S}\mathcal{R}_p^\sigma(f)=\min_{S\in\mathcal{F}_{k}(\mathcal{S}_p)}\max\limits_{f\in S}\mathcal{R}_p^\sigma(f)$$ is continuous  with respect to   $(w,\kappa,\mu)$. 

First, we prove the upper semi-continuity of the variational multiplicity. Let $r$ be the variational multiplicity of $\lambda_k(\Delta_p^\sigma[w_0,\kappa_0,\mu_0])$, where $(w_0,\kappa_0,\mu_0)\in (0,+\infty)^E\times\R^V\times (0,+\infty)^V$. 
Without loss of generality, we assume that
\[\lambda_{k-1}(\Delta_p^\sigma[w_0,\kappa_0,\mu_0])<\lambda_k(\Delta_p^\sigma[w_0,\kappa_0,\mu_0])\,\,\text{and}\,\,\lambda_{k+r-1}(\Delta_p^\sigma[w_0,\kappa_0,\mu_0])<\lambda_{k+r}(\Delta_p^\sigma[w_0,\kappa_0,\mu_0]).\]
By continuity, the above two inequalities hold in an open neighborhood $U$ of $(w_0,\kappa_0,\mu_0)$. Therefore, the variational multiplicity of $\lambda_k(\Delta_p^\sigma[w,\kappa,\mu])$ with $(w,\kappa,\mu)\in U$ is equal to or less than $r$. This proves the upper semi-continuity of the variational multiplicity.

Next, we prove the upper semi-continuity of the multiplicity. Let $X_k(w,\kappa,\mu)\subset\mathcal{S}_p$ be the collection of all normalized eigenfunctions corresponding to the $k$-th variational eigenvalue of $\Delta_p^\sigma$ with the parameter $(w,\kappa,\mu)$. We first verify  that  the eigenspace $X_k(w,\kappa,\mu)$ is upper hemi-continuous with respect to $(w,\kappa,\mu)$. Suppose the contrary, that there exists $\epsilon_0>0$ such that there exists a sequence $(\omega^i,\kappa^i,\mu^i)$ converges  to $(\omega,\kappa,\mu)$ as  $i\to+\infty$, but  $X_k(\omega^i,\kappa^i,\mu^i)$ is not included in the $\epsilon_0$-neighbourhood $\mathbb{B}_{\epsilon_0}(X_k(\omega,\kappa,\mu))$ of $X_k(\omega,\kappa,\mu)$. 
That is, we can take $f^i\in X_k(\omega^i,\kappa^i,\mu^i)\setminus \mathbb{B}_{\epsilon_0}(X_k(\omega,\kappa,\mu))$, for any $i\ge1$. 
Since $f^i\in\mathcal{S}_p$, by the compactness, there exists a subsequence, still denoted  by  $\{f^i\}_{i\ge 1}$, converging to a limit $f\in\mathcal{S}_p\setminus \mathbb{B}_{\epsilon_0}(X_k(\omega,\kappa,\mu))$. Since $f^i$ is an eigenfunction corresponding to the $k$-th variational eigenvalue $\lambda_k(\Delta_p^\sigma[\omega^i,\kappa^i,\mu^i])$ of  $\Delta_p^\sigma[\omega^i,\kappa^i,\mu^i]$, we have the  eigen-equation  
\[\Delta_p^\sigma[\omega^i,\kappa^i,\mu^i] f^i=\lambda_k(\Delta_p^\sigma[\omega^i,\kappa^i,\mu^i])\cdot \mu^i\Phi_p(f^i).\]  
Taking $i\to+\infty$,  we have $\Delta_p^\sigma[\omega,\kappa,\mu] f=\lambda_k(\Delta_p^\sigma[\omega,\kappa,\mu])\cdot \mu\Phi_p(f)$, which  implies that $f$ is an eigenfunction corresponding to $\lambda_k(\Delta_p^\sigma[\omega,\kappa,\mu])$. Thus, $f\in X_k(\omega,\kappa,\mu)$, which contradicts to the fact $f\in\mathcal{S}_p\setminus \mathbb{B}_{\epsilon_0}(X_k(\omega,\kappa,\mu))$. 

By the monotonicity and continuity of the index function  $\gamma$ (\cite[Proposition 5.4]{Struwe}), we have
\begin{align*}
\mathrm{multi}(\lambda_k(\Delta_p^\sigma[w',\kappa',\mu']))&=\gamma(X_k(w',\kappa',\mu'))\le \gamma(\mathbb{B}_\epsilon(X_k(w,\kappa,\mu)))= \gamma(X_k(w,\kappa,\mu))\\&=\mathrm{multi}(\lambda_k(\Delta_p^\sigma[w,\kappa,\mu])),
\end{align*}
 where $\mathbb{B}_\epsilon(X_k(w,\kappa,\mu))$ represents the $\epsilon$-neighborhood of  $X_k(w,\kappa,\mu)$, and $\mathrm{multi}(\lambda_k(\Delta_p^\sigma[w,\kappa,\mu]))$ indicates the multiplicity of the $k$-th variational eigenvalue of $\Delta_p^\sigma$ with the parameter $(w,\kappa,\mu)$. This implies that  $\{(w,\kappa,\mu):\mathrm{multi}(\lambda_k(\Delta_p^\sigma[w,\kappa,\mu]))\le C\}$ is an open subset of $(0,+\infty)^E\times\R^V\times (0,+\infty)^V$, for any constant $C\ge1$. Hence the multiplicity  of the $k$-th variational eigenvalue is upper semi-continuous with respect to $(w,\kappa,\mu)$.
 \end{proof}
 In the linear case, we know that if $(G,\sigma)$ and $(G,\tilde{\sigma})$ are switching equivalent with the same edge measure, vertex weight and  potential function, then the spectrum of $\Delta_2^{\sigma}$ coincides with that of $\Delta_2^{\tilde{\sigma}}$. The following proposition shows this fact still holds for the nonlinear case.
\begin{pro}\label{pro:switching-spectra}
Let $(G,\sigma)$ and $(G,\tilde{\sigma})$ be two signed graphs with the same edge measure, vertex weight and  potential function. If $\tilde{\sigma}$ is  switching equivalent to $\sigma$, then the spectrum of $\Delta_p^{\tilde{\sigma}}$  coincides with the spectrum of $\Delta_p^\sigma$. Moreover, the variational spectra of $\Delta_p^{\tilde{\sigma}}$ and $\Delta_p^\sigma$ are the same.
\end{pro}
\begin{proof}
Suppose  $\tilde{\sigma}:=\sigma^\tau$ for some switching function $\tau:V\to\{-1,+1\}$. By direct computation, we derive that $(\lambda,f)$ is an eigenpair of $\Delta_p^\sigma$ if and only if $(\lambda,\tau f)$ is an eigenpair of $\Delta_p^{\sigma^\tau}$. Therefore, the set of eigenvalues of $\Delta_p^\sigma$ agrees with  the set of eigenvalues of  $\Delta_p^{\sigma^\tau}$. 

Note that for any centrally symmetric subset $X\subset \mathcal{S}_p$  of index $k$, $\tau\cdot X:=\{\tau f:f\in X\}$ is also a centrally symmetric subset of index $k$. For any eigenvalue $\lambda$ of $\Delta_p^{\sigma}$, it is clear that $\tau\cdot\mathsf{X}_{\lambda}(\Delta_p^\sigma)$ is nothing but the collection $\mathsf{X}_{\lambda}(\Delta_p^{\sigma^\tau})$ of the eigenfunctions corresponding to the eigenvalue $\lambda$ of $\Delta_p^{\sigma^\tau}$. Hence, the multiplicity of the eigenvalue $\lambda$ of $\Delta_p^{\sigma}$ coincides with the multiplicity of the eigenvalue $\lambda$ of  $\Delta_p^{\sigma^\tau}$. 
In summary, we obtain that the spectra of $\Delta_p^{\sigma}$ and $\Delta_p^{\sigma^\tau}$ coincide. 

Finally, we focus on the variational eigenvalues. It is direct  to check that $\gamma(A)=\gamma(\tau\cdot A)$ for any centrally symmetric subset $A$. And for any minimizing set $A$ with respect to $\lambda_k(\Delta_p^{\sigma^\tau})$, $\tau\cdot A$ is a minimizing set with  respect to $\lambda_k(\Delta_p^{\sigma})$. It then follows from the fact $\RQ_p^{\sigma^\tau}(f)=\RQ_p^{\sigma}(\tau f)$ that  $\lambda_k(\Delta_p^{\sigma^\tau})=\lambda_k(\Delta_p^{\sigma})$. 
\end{proof}

\section{Nodal domain theorems}
\label{sec:upper-nodal}
In this section, we prove nodal domain theorems for $p$-Laplacians on signed graphs and discuss several applications.
Let $\Gamma=(G,\sigma)$ be a signed graph with $G=(V, E)$, and let  
\begin{equation*}
    \lambda_1\leq \lambda_2\leq \ldots\leq\lambda_{|V|-1}\leq \lambda_{|V|} 
\end{equation*}
be the variational eigenvalues of $\Delta_p^{\sigma}$. For ease of notation, we denote $n=|V|$.

For any eigenfunction $f$ corresponding to $\lambda$,  we prove the following upper bounds for the quantities $\mathfrak{S}(f),\,\mathfrak{W}(f),\, \overline{\mathfrak{S}}(f) \text{ and }\overline{\mathfrak{W}}(f)$. 
\begin{theorem}
\label{thm:nodal-signed-graph}
  For $p\geq 1$, if $\lambda<\lambda_{k+1}$, then we have $\mathfrak{W}(f) \leq \mathfrak{S}(f)\le k.$ 
 \end{theorem}
 \begin{theorem}\label{thm:nodal-anti-graph}
     For $p\geq 1$, if $\lambda>\lambda_{k}$, then we have $\overline{\mathfrak{W}}(f)\leq  \overline{\mathfrak{S}}(f)\le n-k.$      
 \end{theorem}

\begin{theorem}\label{thm:weak nodal domain}
     
       For $p>1$, if $\lambda=\lambda_k$ and 
     \begin{equation*}
         \lambda_1\leq \ldots\leq\lambda_{k-1}<\lambda_k=\lambda_{k+1}=\ldots=\lambda_{k+r-1}<\lambda_{k+r}\leq \ldots \leq\lambda_n,
     \end{equation*}
       then we have 
       \[\mathfrak{W}(f)\leq k+c-1\,\,\text{ and }\,\, \overline{\mathfrak{W}}(f)\leq n-k-r+c+1,\] where $c$ is the number of connected components of $G$.  
\end{theorem}
\begin{theorem}\label{thm:min nodal domain}

       For $p\geq 1$, if $\lambda=\lambda_k$ where 
     \begin{equation*}
         \lambda_1\leq \ldots\leq\lambda_{k-1}<\lambda_k=\lambda_{k+1}=\ldots=\lambda_{k+r-1}<\lambda_{k+r}\leq \ldots \leq\lambda_n,
     \end{equation*}
      and the corresponding eigenfunction $f$ has minimal support, then we have 
      \[\mathfrak{S}(f)\leq k\,\,\text{ and }\,\, \overline{\mathfrak{S}}(f)\leq n-k-r+2.\]  In addition, when $p=1$, and $f$ has minimal support, we further have that  $\mathfrak{S}(f)=1$. Moreover, when the graph is balanced, the number of zeros of $f$ is at least $k+r-2$.
\end{theorem}

Let us first remark on the estimates of $\overline{\mathfrak{S}}(f)$ (resp., $\overline{\mathfrak{W}}(f)$), i.e.,   the number of strong (resp., weak) nodal domains of $f$ with respect to $(G,-\sigma)$. In the linear case, if $f$ is an eigenfunction of the signed Laplacian $\Delta_2^\sigma$ corresponding to $\lambda$, then it is also an eigenfunction of $-\Delta_2^\sigma$ corresponding to $-\lambda$. Since $-\Delta_2^\sigma$ can be considered as a signed Laplacian  of the graph $(G,-\sigma)$ (with a suitable choice of the potential function), the upper bound estimates of $\overline{\mathfrak{S}}(f)$ and $\overline{\mathfrak{W}}(f)$ follows directly from the signed  nodal domain theorem \cite[Theorem 4.1]{GL21+}. However, in the non-linear case, when $f$ is an eigenfunction of $\Delta_p^{\sigma}$, $f$ may not be an eigenfunction of $\Delta_p^{-\sigma}$ anymore. It is an interesting question to ask whether there are still upper bound estimates of $\overline{\mathfrak{S}}(f)$ and $\overline{\mathfrak{W}}(f)$ or not. Theorem \ref{thm:nodal-anti-graph}, Theorem \ref{thm:weak nodal domain} and Theorem \ref{thm:min nodal domain} above answer this question positively. Intriguingly, these upper bound estimates will be very useful in the proofs of our later results, including Theorem \ref{thm:non-var}, Theorem \ref{thm:forest} and Theorem \ref{thm:pr}.

Those above upper bounds can be regarded as discrete versions of the Courant's nodal domain theorem \cite{Courant23,CD53} proved in 1920s. Cheng \cite{Cheng76} studied Courant's theorem on Riemannian manifolds. 
The study of discrete nodal domain theorems for linear Laplacians on graphs dates back to the work of 
Gantmacher and Krein \cite{GK50} in 1940s and the work of Fiedler \cite{Fiedler73,Fiedler75,Fiedler752} in 1970s. Van der Holst \cite{vdH95,vdH96} proved that the second eigenfunction $f_2$ induces $2$ strong nodal domains if it  has minimal support. Duval and Reiner \cite{DR99} studied the   discrete nodal theorems of higher eigenfunctions. In 2001,  Davies, Gladwell, Leydold and Stadler \cite{DGLS01} established the discrete nodal domain theorems for generalized Laplacians. There are amount of works about discrete nodal
domain theorems for linear Laplacians, see, e.g., \cite{Berkolaiko08,Biyikoglu03,CdV93,Friedman93,Lovasz21,Powers88,Roth89}. The extensions to linear Laplacians on signed graphs have been discussed in \cite{Ali16,JMZ,GL21+}, while the extensions to non-linear Laplacians on graphs have been carried out in \cite{CSZ17,DPT21,TudiscoHein18}.


Those above results unify many results on the upper bounds of the number of nodal domains for $p$-Laplacians on graphs and signed graphs,  including  \cite[Theorem 4.1]{GL21+}, \cite[Theorem 5.4]{JMZ} for signed graphs, \cite[Theorem 3.4 and Theorem 3.5]{TudiscoHein18} for graphs. Moreover, the inequality  (see  \cite[Theorem 5.3 ]{JMZ} and  \cite[Theorem 2.2]{JZ21-PM}) \[\mathfrak{N}(f)\le\min\{k+r-1,n-k+r\},\] where $\mathfrak{N}(f)$ stands for the number of connected components of the support of $f$, becomes a direct consequence of these results, since we have $\mathfrak{N}(f)\le \min\{\mathfrak{S}(f),\overline{\mathfrak{S}}(f)\}$. 

We further point out that Theorem \ref{thm:weak nodal domain} can not hold for the case $p=1$, even for balanced signed graphs. A counterexample is given in \cite[Example 10]{CSZ17}.

For the proofs of these theorems, we prepare two lemmas.
The first one has been established in \cite{Amghibech,TudiscoHein18,JMZ}.

\begin{lemma}\label{lemma:elementary}
Let $t,s,a,b$ be real numbers. Then, we have for $p>1$
\begin{equation*}\label{eq:basic-inequality-p}
   \begin{cases}
|ta+sb|^p\ge (|t|^pa+|s|^pb)|a+b|^{p-2}(a+b),&\text{ if }  ab\le 0,\\
|ta+sb|^p\le (|t|^pa+|s|^pb)|a+b|^{p-2}(a+b),&\text{ if }  ab\ge 0.
   \end{cases} 
\end{equation*}
Moreover, the equality holds if and only if 
\begin{equation}\label{eq:lemma1=}
    ab=0\,\,\text{or}\,\, t=s
\end{equation}
in both cases. 

In the case of $p=1$, we have for any $ z\in\mathrm{Sgn}(a+b)$,  
$$\begin{cases}
|ta+sb|\ge (|t|a+|s|b)z,&\text{ if }ab\le 0,\\
|ta+sb|\le (|t|a+|s|b)z,&\text{ if }ab\ge 0.
\end{cases}$$
\end{lemma}
For any function $g:V\to \mathbb{R}$, we define 
$\Vert g\Vert_p^p=\sum_{x\in V}|g(x)|^p\mu_x$ for $p\geq 1$. We will use the notation  $\sum_{i\neq j}:=\sum_i\sum_{j:j\neq i}$ for simplicity. 
\begin{lemma}\label{lemma:Gij}
    For $p\geq 1$, let $f$ be an eigenfunction of $\Delta_p^\sigma$ corresponding to an  eigenvalue $\lambda$. Set $Z:=\{x\in V: f(x)=0\}$. Let $V_1, \ldots, V_m$ be a partition of $V\setminus Z$. Let $X$ be the linear function-space spanned by $f_1,\ldots, f_m$ where 
    \begin{equation*}
    f_i(x)=\begin{cases}f(x),&\text{ if }x\in V_i,\\ 0,&\text{ if } x\not\in V_i.\end{cases}
\end{equation*}
Then, for any $g=\sum_{i=1}^mt_if_i\in X\setminus \vec0$, we have 
\begin{equation}
    \left(\mathcal{R}_p^\sigma(g)-\lambda\right)\Vert g\Vert_p^p=\frac{1}{2}\sum_{i\neq j}\sum_{x\in V_i}\sum_{y\in V_j}w_{xy}G_{ij}(x,y),
\end{equation}
where 
\begin{equation*}
    G_{ij}(x,y)=|t_if_i(x)-\sigma_{xy}t_jf_j(y)|^p-\big(|t_i|^pf_i(x)-\sigma_{xy}|t_j|^pf_j(y)\big) \Phi_p(f_i(x)-\sigma_{xy}f_j(y))
\end{equation*}
if $p>1$, and if $p=1$,
\begin{equation*}
    G_{ij}(x,y)=|t_if_i(x)-\sigma_{xy}t_jf_j(y)|-\left(|t_i|f_i(x)-\sigma_{xy}|t_j|f_j(y)\right)z_{xy},
\end{equation*}
for some $z_{xy}\in \mathrm{Sgn}(f(x)-\sigma_{xy}f(y))$.
\end{lemma}
\begin{proof}
We first compute for any $p\geq 1$ that
\begin{equation}\label{eq:prayleigh-up}
	\begin{aligned}
&\sum_{\{x,y\}\in E}w_{xy}\left|g(x)-\sigma_{xy}g(y)\right|^p+\sum_{x\in V}\kappa_x|g(x)|^p
\\=&\sum_{i=1}^{m} \sum_{x\in V_i}\sum_{y\in Z}w_{xy}\left|t_if_i(x)\right|^p+\frac{1}{2}\sum_{i=1}^{m}\sum_{x\in V_i}\sum_{y\in V_i}w_{xy}|t_i|^p\left |f_i(x)-\sigma_{xy}f_i(y)\right|^p 
\\&+\frac{1}{2}\sum_{i\neq j}\sum_{x\in V_i}\sum_{y\in V_j}w_{xy}\left |t_if_i(x)-\sigma_{xy}t_jf_j(y)\right|^p+\sum_{x\in V}\kappa_x|g(x)|^p.
\end{aligned}
\end{equation}
We next deal with the case $p>1$.
Employing the eigen-equation, we have for each $i\in \{1,\ldots,m\}$
\begin{equation}\label{eq:prayleigh-down_1}
	\begin{aligned}
\lambda \Vert f_i\Vert^p_p & = \lambda\sum_{x\in V}\mu_x |f_i(x)|^p=\sum_{x\in V} f_i(x)\lambda\mu_x\Phi_p(f(x))=\sum_{x\in V} f_i(x)(\Delta_p^{\sigma}f)(x)
\\&=\sum_{x\in V_i}f_i(x)\sum_{y\sim x}w_{xy}\Phi_p\left(f(x)-\sigma_{xy}f(y)\right)+\sum_{x\in V_i}\kappa_x |f(x)|^p
\\&=\sum_{x\in V_i}\sum_{y\in Z}w_{xy}|f_i(x)|^p+\frac{1}{2}\sum_{x\in V_i}\sum_{y\in V_i}w_{xy}|f_i(x)-\sigma_{xy}f_i(y)|^p
\\&+\sum_{x\in V_i}\sum_{\substack{y\in V_j\\j\neq i}}w_{xy}f_i(x)\Phi_p(f_i(x)-\sigma_{xy}f_j(y))+\sum_{x\in V_i}\kappa_x |f_i(x)|^p.
\end{aligned}
\end{equation}
Consequently, we obtain
\begin{equation}\label{eq:prayleigh-down_2}
	\begin{aligned}
\lambda\Vert g \Vert^p_p&=\lambda\sum_{i=1}^{m}|t_i|^p\Vert f_i\Vert^p_p
\\&=\sum_{i=1}^{m}\sum_{x_i\in V_i}\sum_{y\in Z}w_{xy}|t_i|^p|f_i(x)|^p+\frac{1}{2}\sum_{i=1}^{m}\sum_{x\in V_i}\sum_{y\in V_i}w_{xy}|t_i|^p|f_i(x)-\sigma_{xy}f_i(y)|^p
\\&+\frac{1}{2}\sum_{i\neq j}\sum_{x\in V_i}\sum_{y \in V_j}w_{xy}\left(|t_i|^pf_i(x)-\sigma_{xy}|t_j|^pf_j(y) \right)\Phi_p(f_i(x)-\sigma_{xy}f_j(y)) +\sum_{i=1}^{m}\sum_{x\in V_i}\kappa_x|t_i|^p |f(x)|^p.
\end{aligned}
\end{equation}
Combining (\ref{eq:prayleigh-up}) and (\ref{eq:prayleigh-down_2}), we get
\begin{equation*}
   \sum_{\{x,y\}\in E}w_{xy}\left|g(x)-\sigma_{xy}g(y)\right|^p+\sum_{x\in V}\kappa_x|g(x)|^p-\lambda\Vert g \Vert_p^p=\frac{1}{2}\sum_{i\neq j}\sum_{x\in V_i}\sum_{y\in V_j}w_{xy}G_{ij}(x,y), 
\end{equation*}
where 
\begin{equation*}
    G_{ij}(x,y)=|t_if_i(x)-\sigma_{xy}t_jf_j(y)|^p-\big(|t_i|^pf_i(x)-\sigma_{xy}|t_j|^pf_j(y)\big) \Phi_p(f_i(x)-\sigma_{xy}f_j(y)).
\end{equation*}
This completes the proof for the case $p>1$. 

Finally, we discuss the case $p=1$. By definition, we have 
\[\Delta_1^\sigma f(x)\bigcap \lambda \mu_x \mathrm{Sgn}(f(x))\ne\emptyset,\] 
for any $x \in V$. Hence, there exist $z_{xy}\in \mathrm{Sgn}\big(f(x)-\sigma_{xy}f(y)\big),\,z_{xy}=-\sigma_{xy}z_{yx}$, $z_x\in \mathrm{Sgn}(f(x))$ and $z'_x\in \mathrm{Sgn}(f(x))$ such that $\sum_{y\sim x}w_{xy}z_{xy}+k_x z_x =\lambda \mu_x z'_x$, for any $x\in V$. For any $i\in \{1,\ldots,m\}$, we compute
\begin{equation}\label{eq:1rayleigh-down-1}
\begin{aligned}
\lambda \Vert f_i\Vert_1 & = \lambda\sum_{x\in V}\mu_x |f_i(x)|=\sum_{x\in V} f_i(x)\lambda\mu_x z'_x
\\&=\sum_{x\in V}f_i(x)\left(\sum_{y\sim x}w_{xy}z_{xy}+k_x z_x    \right)
\\&=\sum_{x\in V_i}f_i(x)\sum_{y\sim x}w_{xy}z_{xy}+\sum_{x\in V_i}\kappa_x |f(x)|
\\&=\sum_{x\in V_i}\sum_{y\in Z}w_{xy}|f(x)|+\frac{1}{2}\sum_{x\in V_i}\sum_{y\in V_i}w_{xy}|f_i(x)-\sigma_{xy}f_i(y)|
\\&+\sum_{x\in V_i}\sum_{\substack{y\in V_j\\j\neq i}}w_{xy}f_i(x)z_{xy}+\sum_{x\in V_i}\kappa_x |f(x)|.
\end{aligned}
\end{equation}
Consequently, we derive
\begin{equation}\label{eq:1rayleigh-down-2}
\begin{aligned}
\lambda\Vert g \Vert_1&=\lambda\sum_{i=1}^{m}|t_i|\Vert f_i\Vert_{1}
\\&=\sum_{i=1}^{m}\sum_{x_i\in V_i}\sum_{y\in Z}w_{xy}|t_i f_i(x)|+\frac{1}{2}\sum_{i=1}^{m}\sum_{x\in V_i}\sum_{y\in V_i}w_{xy}|t_i||f_i(x)-\sigma_{xy}f_i(y)|
\\&+\frac{1}{2}\sum_{i\neq j}\sum_{x\in V_i}\sum_{y \in V_j}w_{xy}\left(|t_i|f_i(x)-\sigma_{xy}|t_j|f_j(y) \right)z_{xy} +\sum_{i=1}^{m}\sum_{x\in V_i}\kappa_x|t_i| |f(x)|.
\end{aligned}
\end{equation}
Combining (\ref{eq:prayleigh-up}) and (\ref{eq:1rayleigh-down-2}) yields
\begin{equation*}
   \sum_{\{x,y\}\in E}w_{xy}\left|g(x)-\sigma_{xy}g(y)\right|+\sum_{x\in V}\kappa_x|g(x)|-\lambda\Vert g \Vert_1=\frac{1}{2}\sum_{i\neq j}\sum_{x\in V_i}\sum_{y\in V_j}w_{xy}G_{ij}(x,y), 
\end{equation*}
where 
\begin{equation*}
    G_{ij}(x,y)=|t_if_i(x)-\sigma_{xy}t_jf_j(y)|-\left(|t_i|f_i(x)-\sigma_{xy}|t_j|f_j(y)\right)z_{xy}.
\end{equation*}
This completes the proof for the case $p=1$.
\end{proof}
We are now well-prepared for the proof of Theorem \ref{thm:nodal-signed-graph}.
\begin{proof}[Proof of Theorem \ref{thm:nodal-signed-graph}]
 By definition, we have $\mathfrak{W}(f)\leq \mathfrak{S}(f)$. Next, we prove $\mathfrak{S}(f)\leq k.$
 
 Suppose that $f$ has $m$ strong nodal domains on $\Gamma=(G,\sigma)$ which are denoted by $V_1,\ldots,V_{m}$.  
Consider the linear function-space $X$ spanned by $f_1,\ldots,f_m$, where  $f_i$ is defined by
\begin{equation*}
    f_i(x)=\begin{cases}f(x),&\text{ if }x\in V_i,\\ 0,&\text{ if } x\not\in V_i.\end{cases}
\end{equation*}
Since $V_1,\ldots,V_{m}$ are pairwise disjoint, we have $\dim X=m$. Then we can use Proposition \ref{index} to get 
\[\gamma(X\cap \mathcal{S}_p)=m.\] 

We claim that  $\RQ_p^{\sigma}(g)\leq \lambda$ for any $g=\sum_{i=1}^{m} t_i f_i\in X\setminus \vec0$. Indeed, we have by Lemma \ref{lemma:Gij}, 
\[ \left(\mathcal{R}_p^\sigma(g)-\lambda\right)\Vert g\Vert_p^p=\frac{1}{2}\sum_{i\neq j}\sum_{x\in V_i}\sum_{y\in V_j}w_{xy}G_{ij}(x,y). \]
For any $i\neq j$, $x\in V_i$ and $y\in V_j$, we take $a=f_i(x),b=-\sigma_{xy}f_j(y)$, $t=t_i$ and $s=t_j$. Because $x$ and $y$ lie in different strong nodal domains, we have $ab=-f_i(x)\sigma_{xy} f_j(y)> 0$. Then we use  Lemma \ref{lemma:elementary} to get $G_{ij}(x,y)\leq 0$.
That is, we have $\RQ_p^{\sigma}(g)\leq \lambda$.

By definition, we have 
$$\lambda_{m}=\inf\limits_{X'\in \mathcal{F}_m(\mathcal{S}_p) }\sup\limits_{g'\in X'}
\RQ_p^\sigma(g')\le \sup\limits_{g\in X\cap \mathcal{S}_p}
\RQ_p^\sigma(g)\leq \lambda< \lambda_{k+1}.$$ This implies 
$m\leq  k$. 
\end{proof}

In order to prove the upper bound of $\overline{\mathfrak{S}}(f)$ in Theorem \ref{thm:nodal-anti-graph}, we recall the following lemma from \cite[Proposition 4.2.20]{Papageorgiou09}.
\begin{lemma}[\cite{Papageorgiou09}]\label{lemma:intersect}
    If $X$ is a Banach space, $Y$ is a finite-dimensional linear subspace of $X$, $p_Y\in \mathcal{L}(X)$ is the projection operator onto $Y$, and $A$ is a closed  centrally symmetric subset with $\gamma(A)>k=\dim(Y)$, then $A\cap(Id-p_Y)(X)\neq \emptyset$.
\end{lemma}

\begin{proof}[Proof of Theorem \ref{thm:nodal-anti-graph}]
     By definition, we have $\overline{\mathfrak{W}}(f)\leq\overline{ \mathfrak{S}}(f)$. Next, we prove $\overline{\mathfrak{S}}(f)\leq n-k.$
    
As above, we suppose that $f$ has $m$ strong nodal domains on $\Gamma'=(G,-\sigma)$ which are denoted by $\overline{V}_1,\ldots,\overline{V}_{m}$.   Let  $\overline{X}$ be the linear function-space spanned by $f_1,\ldots,f_m$,
    where $f_i$ is defined as follows
\begin{equation*}
    f_i(x)=\begin{cases}f(x),&\text{ if }x\in \overline{V}_i,\\ 0,&\text{ if } x\not\in \overline{V}_i.\end{cases}
\end{equation*}

 We first prove that $\RQ_{p}^{\sigma}(g)\ge\lambda$ for any $g=\sum_{i=1}^nt_if_i\in \overline{X}\setminus \vec0$. Indeed, we have by Lemma \ref{lemma:Gij}, 
\[ \left(\mathcal{R}_p^\sigma(g)-\lambda\right)\Vert g\Vert_p^p=\frac{1}{2}\sum_{i\neq j}\sum_{x\in \overline{V}_i}\sum_{y\in \overline{V}_j}w_{xy}G_{ij}(x,y). \]
For any $i\neq j$, $x\in \overline{V}_i$ and $y\in \overline{V}_j$, we take $a=f_i(x),b=-\sigma_{xy}f_j(y)$ and $t=t_i,s=t_j$. Because $x$ and $y$ lie in different strong nodal domains on $\Gamma=(G,-\sigma)$, we have by definition $ab=-f_i(x)\sigma_{xy} f_j(y)<0$. Then we use  Lemma  \ref{lemma:elementary}  to get $G_{ij}(x,y)\geq 0$. That is, we have $\RQ_{p}^{\sigma}(g)\ge\lambda$.

Notice that, by Lemma \ref{lemma:intersect}, $X'\cap \overline{X}\ne\emptyset$ for any $X'\in \mathcal{F}_{n-m+1}(\mathcal{S}_p)$. Then we have by definition 
\begin{align*}
\lambda_{n-m+1}&=\inf\limits_{ X'\in \mathcal{F}_{n-m+1}( \mathcal{S}_p)}\sup\limits_{g'\in X'}\RQ_p^\sigma(g') \ge \inf\limits_{ X'\in \mathcal{F}_{n-m+1}( \mathcal{S}_p)}\inf\limits_{g'\in X'\cap \overline{X}}\RQ_p^\sigma(g') \\&\ge \inf\limits_{g\in \overline{X}\setminus \vec0}\RQ_p^\sigma(g) \ge\lambda>\lambda_{k},
\end{align*}
which implies $n-m+1>k$, i.e., $m\le n-k$. This completes the proof.
\end{proof}

To show the upper bounds of $\mathfrak{W}(f)$ and $\overline{\mathfrak{W}}(f)$ in Theorem \ref{thm:weak nodal domain}, we prepare the following two lemmas: The first one is a reformulation of a related result by Hein and  Tudisco  \cite[Lemma 2.3]{TudiscoHein18}; The second one is a new result for estimating the number of dual nodal domains. It is worth noting that any $f\in \mathcal{S}_p$ is a critical point of $\mathcal{R}_p^{\sigma}$ corresponding to $\lambda_k$ if and only if it is an eigenfunction of $\Delta_p^{\sigma}$ corresponding to $\lambda_k$.
\begin{lemma}\label{lemma:exsit1}
    For $p\ge 1$ and $k\geq 1$, let $A^*\in \mathcal{F}_k(\mathcal{S}_p)$ be such that  
\begin{equation*}
    \lambda_k=\inf\limits_{ A\in \mathcal{F}_{k}( \mathcal{S}_p)}\sup\limits_{g\in A}\RQ_p^\sigma(g)=\sup\limits_{g\in A^*}\RQ_p^{\sigma}(g).
\end{equation*}
Then $A^*$ contains at least one critical point of $\mathcal{R}_p^{\sigma}$ corresponding to $\lambda_k$.
\end{lemma}

\begin{proof}
The proof follows the same line of that of  \cite[Lemma 2.3]{TudiscoHein18}, with the only difference being that the deformation lemma  is used to  construct an odd continuous map  to deform the minimizing set $A^*$.
\end{proof}

\begin{lemma}\label{lemma:exsit}
    For $p\ge1$ and $k\geq 1$, let $X$ be a linear subspace of dimension $n-k+1$ such that  
\begin{equation*}
    \lambda_k=\inf\limits_{A\in \mathcal{F}_{k}( \mathcal{S}_p)}\sup\limits_{g\in A}\RQ_p^\sigma(g)=\inf\limits_{g\in X\setminus \vec0}\RQ_p^{\sigma}(g)=\min\limits_{g\in X\cap\mathcal{S}_p }\RQ_p^{\sigma}(g).
\end{equation*}
Then $X\cap \mathcal{S}_p$ contains as least one  critical point of $\mathcal{R}_p^{\sigma}$   corresponding to $\lambda_k$.
\end{lemma}
\begin{proof}
We first concentrate on the case of $p>1$.  Suppose the contrary, that 
$X\cap \mathcal{S}_p$ has no critical points of $\RQ_p^\sigma$ corresponding to $\lambda_k$. Let $K_{\lambda_k}(\RQ_p^\sigma)$ be the set consists of all critical points in $\mathcal{S}_p$ of $\mathcal{R}_p^{\sigma}$ corresponding to  $\lambda_k$. By definition, we know $K_{\lambda_k}(\RQ_p^\sigma)$ is closed. By assumption, we have $X\cap \mathcal{S}_p\cap K_{\lambda_k}(\RQ_p^\sigma)=\emptyset$. Then there exists a neighborhood of $K_{\lambda_k}(\RQ_p^\sigma)$ denoted by $N(K_{\lambda_k}(\RQ_p^\sigma))$ such that $X\cap\mathcal{S}_p\cap N(K_{\lambda_k}(\RQ_p^\sigma))=\emptyset$.
Since $p>1$,  $\mathcal{S}_p$ is a $C^{1,1}$ manifold and $\RQ_p^{\sigma}$ is smooth, we can apply \cite[Theorem 3.11]{Struwe} to derive that  there exists an odd homeomorphism $\theta:\mathcal{S}_p\to \mathcal{S}_p$ with $$\theta(\{g\in\mathcal{S}_p:\RQ_p^\sigma(g)\ge \lambda_k-\epsilon\}\setminus N(K_{\lambda_k}(\RQ_p^\sigma)))\subset \{g\in\mathcal{S}_p:\RQ_p^\sigma(g)\ge \lambda_k+\epsilon\},$$ where  $\epsilon>0$ is sufficiently small. 
In particular, we 
 have \[\theta(\mathcal{S}_p\cap X)\subset \{g\in\mathcal{S}_p:\RQ_p^\sigma(g)\ge \lambda_k+\epsilon\} \setminus N(K_{\lambda_k}(\RQ_p^\sigma)).\] 
Let $A$ be a  minimizing set corresponding to  $\lambda_k$. We have $\gamma(A)\ge k$. Since $\theta$ is an odd homeomorphism, the inverse map $\theta^{-1}$ is odd continuous. By the continuity  property of the index function $\gamma$,  we have  $\gamma(\theta^{-1}(A))\ge k$. 
So, by the intersection property of the index function $\gamma$ (see also Lemma \ref{lemma:intersect}),  $\theta^{-1}(A)\cap \mathcal{S}_p\cap X\ne\emptyset$. Thus $A\cap\theta(\mathcal{S}_p\cap X)=\theta(\theta^{-1}(A)\cap \mathcal{S}_p\cap X)\ne\emptyset$. Then, we obtain $$\lambda_k=\sup_{g\in A}\RQ_p^\sigma(g)\ge \min_{g\in \theta(\mathcal{S}_p\cap X)}\RQ_p^\sigma(g)\ge \lambda_k+\epsilon $$ which is a contradiction.

For the case of $p=1$, we consider the restriction $\RQ_1^{\sigma}|_{\mathcal{S}_2}$. 
Then \cite[Remark 3.3]{Chang81} implies that the generalized Clarke  gradient $\partial \RQ_1^{\sigma}|_{\mathcal{S}_2}(g)$ restricted on $\mathcal{S}_2$ is the set $\{h-\langle h,g\rangle g:h\in \partial \RQ_1^{\sigma}(g)\}$. 
By \cite[Proposition 2.3.14]{Clarke}, the Clarke derivative satisfies $$\partial \RQ_1^{\sigma}(g)\subset \frac{1}{\|g\|_{1}}\left(\partial \mathrm{TV}(g)-\frac{\mathrm{TV}(g)}{\|g\|_{1}}\partial \|g\|_{1}\right)
\subset \frac{1}{\|g\|_{1}}\left(\Delta_1^\sigma g-\frac{\mathrm{TV}(g)}{\|g\|_{1}}\mu \mathrm{Sgn}(g)\right)$$
where $\mathrm{TV}(g):=\sum_{\{x,y\}\in E}w_{xy}|g(x)-\sigma_{xy}g(y)|+\kappa_x|g(x)|$. 
According to the facts $\langle g,\Delta_1^\sigma g\rangle=\mathrm{TV}(g)$ and $\langle g,\mu \mathrm{Sgn}(g)\rangle=\|g\|_{1}$, we have $\langle g,\partial \RQ_1^{\sigma}(g)\rangle= 0$, i.e., $\langle g,h\rangle=0$ for any $h\in \partial \RQ_1^{\sigma}(g)$. 
So, we have \[\partial\RQ_1^{\sigma}|_{\mathcal{S}_2}(g)=\partial \RQ_1^{\sigma}(g), \,\,\text{for any}\,\,g\in \mathcal{S}_2.\]
That is, the set of critical points  of $\RQ_1^{\sigma}$ with $l^2$-norm one coincide with the that of the restriction  $\RQ_1^{\sigma}|_{\mathcal{S}_2}$. 
We then apply \cite[Theorem 3.1, Remarks 3.3 and  3.4]{Chang} to deduce that there is an odd homeomorphism $\theta:\mathcal{S}_2\to \mathcal{S}_2$ with  $$\theta(\{g\in\mathcal{S}_2:\RQ_1^\sigma(g)\ge \lambda_k-\epsilon\}\setminus N(K_{\lambda_k}(\RQ_1^\sigma)))\subset \{g\in\mathcal{S}_2:\RQ_1^\sigma(g)\ge \lambda_k+\epsilon\},$$
where $\epsilon>0$ is sufficiently small.

Let $\eta:\mathcal{S}_1\to \mathcal{S}_2$ be an odd homeomorphism defined as $\eta(f)=f/\|f\|_2$. 
Then, along the line of the proof for the case of $p>1$, we derive for a  minimizing set $A\subset \mathcal{S}_1$  corresponding to $\lambda_k$ that, 
$$\lambda_k=\sup_{g\in A}\RQ_1^\sigma(g)=\sup_{g\in \eta(A)}\RQ_1^\sigma(g)\ge \min_{g\in \theta(\mathcal{S}_2\cap X)}\RQ_1^\sigma(g)
\ge \lambda_k+\epsilon,$$
which is a contradiction.
\end{proof}
\begin{proof}[Proof of Theorem \ref{thm:weak nodal domain}: Upper bound of $\mathfrak{W}(f)$]
      Suppose $f$ has  $m$ weak nodal domains which are denoted by  $U_1,\ldots,U_{m}$. Let $W_1,\ldots,W_c$ be the $c$ connected components of the graph. Then, for any $i\in\{1,\ldots,m\}$, there exists a unique $l\in\{1,\ldots,c\}$ such that $U_i\subset W_l$. For $l=1,\ldots,c$, We denote by 
      \[I_l=\{i\in\{1,\ldots,m\}:U_i\subset W_l\}\]
      the index set corresponding to $W_l$. Then, we have $\bigsqcup_{l=1}^c I_l=\{1,\ldots,m\}$. 

We prove that by contradiction. Assume $m\ge k+c$. Let $X$ be the linear function-space  spanned by $f|_{U_1},\ldots,f|_{U_{m}}$ where $f|_{U_i}=f$ on $ U_i$ and $f|_{U_i}=0$ on $V\setminus U_i$ for any $1\leq i\leq m$. Let $X'$ be the linear function-space  spanned by $f|_{W_1},\ldots,f|_{W_c}$ where $f|_{W_j}=f$ on $ W_j$ and $f|_{W_j}=0$ on $V\setminus W_j$ for any $1\leq j\leq c$. Similarly as the proof of Theorem \ref{thm:nodal-signed-graph}, we drive from Lemma \ref{lemma:elementary} and Lemma \ref{lemma:Gij} that \begin{equation}\label{eq:Rplambda}
    \RQ_p^\sigma(h)\le \lambda_k, \,\,\text{for any}\,\,h\in X\setminus \vec0.
\end{equation}
By definition, we have $f|_{W_l}=\sum_{i\in I_l}f|_{U_i}$, and hence $X'$ is a linear subspace of $X$. We can have a decomposition $X=X'\bigoplus Y$. Since $\dim X=m\geq k+c$ and $\dim X'=c$, we derive $\dim Y \geq k$, and hence, $\gamma(Y\cap \mathcal{S}_p)\geq k$ by Proposition \ref{index}.

According to the definition of variational eigenvalues, there holds
\begin{equation*}
    \lambda_k=\inf\limits_{ Y'\in \mathcal{F}_{k}( \mathcal{S}_p)}\sup\limits_{g'\in Y'}\RQ_p^\sigma(g')\leq \max\limits_{g'\in Y \cap \mathcal{S}_p}\RQ_p^\sigma(g')\leq\lambda_k.
\end{equation*}
So we have $\max\limits_{g'\in Y \cap \mathcal{S}_p}\RQ_p^\sigma(g')=\lambda_k.$ By Lemma \ref{lemma:exsit1}, there exists an eigenfunction $g=\sum_{i=1}^{m} t_i f|_{U_i}\in Y$ corresponding to $\lambda_k$. That is, the equality in \eqref{eq:Rplambda} holds for $h=g=\sum_{i=1}^{m} t_i f|_{U_i}$. 

Let $U_i$ and $U_j$ be two adjacent weak nodal domains. If there exist $x_0\in U_i$  and $y_0\in U_{j}$ such that $\{x_0,y_0\}\in E$, $f(x_0)\neq 0$ and $f(y_0)\neq 0$, then we derive from the condition \eqref{eq:lemma1=} in Lemma \ref{lemma:elementary} that $t_i=t_j$. 
If, otherwise, there exist $x_0\in U_i$ and $y_0\in U_{j}$ such that $\{x_0,y_0\}\in E$ and $f(x_0)=0$, $f(y_0)\neq 0$  or  $f(x_0)\neq 0$, $f(y_0)= 0$, then we claim $t_i=t_j$ still holds. Without loss of generality, we assume $f(x_0)=0$ and $f(y_0)\neq 0$.

Indeed, since $f$ and $g$ are eigenfunctions, we have
    \begin{equation}\label{eq:eigenweak}
    \sum_{y\sim x}w_{x_0y}\Phi_p(\sigma_{x_0y}f(y))=0,\,\,\text{and}\,\,\sum_{y\sim x}w_{x_0y}\Phi_p(\sigma_{x_0y}g(y))=0.
    \end{equation}
We derive from Proposition \ref{pro:cap} that every $y\sim x_0$ lies in either $U_i$ or $U_j$. In fact, if there exists
$y\sim x_0$ such that $y\in U_k$ for some $k\neq i,j$, then we have $x_0\in U_i\cap U_j\cap U_k$ by definition of weak nodal domains and the fact $f(x_0)=0$. This contradicts to Proposition \ref{pro:cap}.  From the equalities in \eqref{eq:eigenweak}, we obtain
\begin{align*}
&\sum_{y\in U_i}w_{x_0y}\Phi_p(\sigma_{x_0y}f(y))+\sum_{y\in U_j}w_{x_0y}\Phi_p(\sigma_{x_0y}f(y))=0,\\
&\Phi_p(t_i)\sum_{y\in U_i}w_{x_0y}\Phi_p(\sigma_{x_0y}f(y))+\Phi_p(t_j)\sum_{y\in U_j}w_{x_0y}\Phi_p(\sigma_{x_0y}f(y))=0,
\end{align*}
and hence,
\begin{equation}\label{eq:Phiptitj}
   \left(\Phi_p(t_i)-\Phi_p(t_j)\right)\sum_{y\in U_j}
   w_{x_0y}\Phi_p(\sigma_{x_0y}f(y))
   =0.
\end{equation} 
By definition of weak nodal domain walk, for any $y,y'\in U_j$ with $\{x_0,y\},\{x_0,y'\}\in E$, we have $(\sigma_{x_0y}f(y))\cdot (\sigma_{x_0y'}f(y'))=f(y)\sigma_{yx_0}\sigma_{x_0y'}f(y')\ge0$ and $f(y_0)\neq 0$, which implies that
\begin{equation*}
  \sum_{y\in U_j}
   w_{x_0y}\Phi_p(\sigma_{x_0y}f(y))=  \sum_{y\in U_j}
   w_{x_0y}|f(y)|^{p-2}(\sigma_{x_0y}f(y))\ne0.
\end{equation*}
Thus, we derive from (\ref{eq:Phiptitj}) that $\Phi(t_i)-\Phi(t_{j})=0$, which  yields  $t_i=t_{j}$. 

In conclusion, we have $t_i=t_{j}$ whenever $U_i$ and $U_{j}$ are adjacent. Thus, in each connected component $W_l$, we use Proposition \ref{pro:connected} to  get  $t_i=t_j$ whenever $i,j\in I_l$. But this implies $g\in X'\setminus \vec0$, which is a contradiction with $g\in Y$. This completes the proof of $\mathfrak{W}(f)\le k+c-1$.
\end{proof}
Next, we prove the upper bound of $\overline{\mathfrak{W}}(f)$.
\begin{proof}[Proof of Theorem \ref{thm:weak nodal domain}: Upper bound of $\overline{\mathfrak{W}}(f)$]
     Suppose $f$ has  $m$ weak nodal domains which are denoted by  $\overline{U}_1,\ldots,\overline{U}_{m}$ with respect to the opposite  signed graph $(G,-\sigma)$. 

Suppose, to the contrary, that $m\geq n-k-r+c+2$. Let $\{\overline{W}_i\}_{i=1}^c$ be the connected components of $G$. For any $1\leq i \leq m$, let $f|_{\overline{U}_i}$ be the function that equals $f$ on $\overline{U}_i$ and zero on $V\setminus \overline{U}_i$. Define $\overline{X}$ to be the linear function-space spanned by $f|_{\overline{U}_1},\ldots,f|_{\overline{U}_m}$. For any $1\leq j\leq c$, let $f|_{\overline{W}_j}$ be the function that equals $f$ on $\overline{W}_j$ and equals zero on $V\setminus \overline{W}_j$. Define $\overline{X}'$ to be the linear function-space spanned by $f|_{\overline{W}_1},\ldots,f|_{\overline{W}_c}$. As above, $\overline{X}'$ is a linear subspace of $\overline{X}$ and we can have a decomposition  $\overline{X}=\overline{X}'\bigoplus \overline{Y}$. Since $\dim \overline{X}\geq n-k-r+c+2$ and $\dim \overline{X}'=c$, we have $\dim  \overline{Y}\geq n-k-r+2$.

Following the same line of the proof of Theorem \ref{thm:nodal-anti-graph}, we drive from Lemma \ref{lemma:elementary} and Lemma \ref{lemma:Gij} that 
\begin{equation}\label{eq:Rpdual}
    \RQ_p^\sigma(h)\ge \lambda_k,\,\,\text{for any}\,\,h\in \overline{X}\setminus\vec0.
\end{equation}
Observe by Lemma \ref{lemma:intersect} that $A\cap  \overline{Y}\neq \emptyset$ for any $A\in \mathcal{F}_{k+r-1} (\mathcal{S}_p)$. Then we prove that
\begin{align*}
\lambda_k=\lambda_{k+r-1}&=\inf\limits_{ A\in \mathcal{F}_{k+r-1}( \mathcal{S}_p)}\sup\limits_{g'\in A}\RQ_p^\sigma(g') \ge \inf\limits_{ A\in \mathcal{F}_{k+r-1}( \mathcal{S}_p)}\inf\limits_{g'\in A\cap \overline{Y}}\RQ_p^\sigma(g') \\&\ge \inf\limits_{g\in \overline{Y}\setminus\vec0}\RQ_p^\sigma(g) \ge \lambda_{k},
\end{align*}
So the above inequalities hold with equalities. In particular,
\begin{equation*}
    \min\limits_{g'\in  \overline{Y}\setminus\vec0}\RQ_p^\sigma(g')=\lambda_k.
\end{equation*}
Then, Lemma \ref{lemma:exsit} implies that there exists an eigenfunction $g=\sum_{i=1}^{m} t_i f|_{\overline{U}_i}\in \overline{Y}$ with $\RQ_p^\sigma(g)=\lambda_k$. That is, the equality in (\ref{eq:Rpdual}) holds for $h=g=\sum_{i=1}^{m} t_i f|_{\overline{U}_i}$.
 
Along the same line of the proof for $\mathfrak{W}(f)\le k+c-1$, we get a contradiction that the nonzero function  $g$ belongs to both $\overline{X}'$ and $\overline{Y}$, which completes the proof.
\end{proof}
In the following, we prove Theorem \ref{thm:min nodal domain}. For the $p=1$ part of Theorem \ref{thm:min nodal domain}, we show the following lemma.

\begin{lemma}[localization property of 1-Laplacian]\label{lem:localization-1}
Let $(\lambda,f)$ be an eigenpair of $\Delta_1^\sigma$. Then, for any strong nodal domain $U$ of $f$, and any $c\ge0$ such that $\{x\in U:f(x)>c\}$ or $\{x\in U:f(x)<-c\}$ is non-empty,  both $f|_U$ and $1_{\{x\in U:f(x)>c\}}-1_{\{x\in U:f(x)<-c\}}$ are eigenfunctions corresponding to the same eigenvalue $\lambda$ of $\Delta_1^\sigma$.

In addition, if $f$ has minimal support, then $f$ has only one strong nodal domain, denoted by $U$, and $f$ must be in the form of $t(1_A-1_B)$ for some $t\ne0$ and some disjoint subsets $A,B$ with $A\cup B=U$. Moreover,  \[\mathsf{X}_{\lambda}(\Delta_1^\sigma)\cap\{g\in C(V):\mathrm{supp}(g)\subset A\cup B\}\subset\{1_{A'}-1_{B'}:A'\cup B'=A\cup B\}\] 
is a finite set with index $1$. 
\end{lemma}
\begin{proof}Set $f_{U,c}:=1_{\{x\in U:f(x)>c\}}-1_{\{x\in U:f(x)<-c\}}$.  First, it is straightforward to verify that 
\begin{align*}
\mathrm{Sgn}(f(x)-\sigma_{xy}f(y))\subset \mathrm{Sgn}(f|_U(x)-\sigma_{xy}f|_U(y))\subset \mathrm{Sgn}(f_{U,c}(x)-\sigma_{xy}f_{U,c}(y))
\end{align*}
 and 
 \[\mathrm{Sgn}(f(x))\subset \mathrm{Sgn}(f|_U(x))\subset  \mathrm{Sgn}(f_{U,c}(x))\]
 for any $x,y\in V$, any $c\ge0$ and any strong nodal domain $U$ of $f$. 
It means that as a set-valued map, $\Delta_1^\sigma f(x)\subset \Delta_1^\sigma f|_U(x)\subset \Delta_1^\sigma f_{U,c}(x)$ for any $x\in V$. 
Since $f$  is an eigenfunction corresponding to an eigenvalue $\lambda$ of $\Delta_1^\sigma$, we have the differential inclusion $$0\in \Delta_1^\sigma f(x)- \lambda \mu_x \mathrm{Sgn}(f(x))\subset \Delta_1^\sigma f|_U(x)- \lambda \mu_x \mathrm{Sgn}(f|_U(x))\subset \Delta_1^\sigma f_{U,c}(x)- \lambda \mu_x \mathrm{Sgn}(f_{U,c}(x)),$$ for any $x\in V$. That is, both $f|_U$ and $f_{U,c}$ are eigenfunctions  corresponding to $\lambda$.

Now, we further assume that $f$ has minimal support. Then, by the localization property proved above, $f$ has only one strong nodal domain, denoted by $U$. 
Suppose, to the contrary, that $f$ is not in the form of $t(1_A-1_B)$. Then there exists $c>0$ such that the support of $f_{U,c}$ is a nonempty proper subset of $U$. So,  we construct an eigenfunction $f_{U,c}$ corresponding to the eigenvalue $\lambda$, but its support is a proper subset of the support of $f$, which leads to a contradiction with the minimal support assumption on $f$. 

Therefore, we have shown that $f$ is in the form of $t(1_A-1_B)$, and its strong nodal domain $U$ is the disjoint union of $A$ and $B$. 
Clearly, for any $g$ whose support is included in $U$, if $g$ is also an eigenfunction corresponding to the eigenvalue $\lambda$, $g=t'(1_{A'}-1_{B'})$ for some $t'\ne0$ and some disjoint subsets $A'$ and $B'$ with $A'\cup B'=U=A\cup B$. That means, $\mathsf{X}_{\lambda}(\Delta_1^\sigma)\cap \{g\in C(V):\mathrm{supp}(g)\subset U\}$ is a finite set, and its index is one. 
\end{proof}

\begin{proof}[Proof of Theorem \ref{thm:min nodal domain}]
Recall we assume that $f$ has minimal support.

We first prove that $\mathfrak{S}(f)\leq k$. Let $\{V_i\}_{i=1}^m$ be the strong nodal domains of $f$ on $\Gamma=(G,\sigma)$. We prove it by contradiction. Assume $m>k$. Consider two linear spaces $X$ and $X'$ defined as follows
\begin{equation*}
	X=\left\{\left.\sum_{i=1}^{m}a_if|_{V_i}\,\right|\,a_i\in \R\right\}\,\,\text{and}\,\,X'=\left\{\left.\sum_{i=1}^{m-1}a_if|_{V_i}\,\right|\,a_i\in \R\right\},
\end{equation*}
where $f|_{V_i}$ is the restriction of $f$ to $V_i$. 
By the proof of Theorem \ref{thm:nodal-signed-graph}, we have 
\begin{equation}
    \mathcal{R}_p^\sigma(g) \leq \lambda_{k}, \,\,\text{for any}\,\,g\in X\setminus\vec0.
\end{equation}

 By Proposition \ref{index}, we have $\gamma(X\cap \mathcal{S}_p)=m>k$ and $\gamma(X'\cap \mathcal{S}_p)=m-1\geq k$.
By definition of variational eigenvalues, we get $$\lambda_{k}=\inf\limits_{A\in \mathcal{F}_k(\mathcal{S}_p)}\sup\limits_{g'\in A}
\mathcal{R}_p^\sigma(g')\le \sup\limits_{g\in X'\cap \mathcal{S}_p}
\mathcal{R}_p^\sigma(g)\le\sup\limits_{g\in X\cap \mathcal{S}_p}
\mathcal{R}_p^\sigma(g) \leq \lambda_{k}.$$
Therefore, all the inequalities above are equalities. In particular, $X'\cap \mathcal{S}_p$ is a minimizing set. By Lemma  \ref{lemma:exsit1}, there exists an eigenfunction $g_0=\sum_{i=1}^{m-1}b_if|_{V_i}$ corresponding to $\lambda$, which contradicts to the fact that $f$ has minimal support. This proves $m\leq k$.

Next, we prove $\overline{\mathfrak{S}}(f)\leq n-k-r+2$. 
Let $\{\overline{V}_i\}_{i=1}^m$ be the strong nodal domains of $f$ with respect to the opposite  signed graph $(G,-\sigma)$.  We prove it by contradiction. Assume $m>n-k-r+2$. Consider two linear spaces $X$ and $X'$ defined as
\begin{equation*}
	X=\left\{\left.\sum_{i=1}^{m}a_if|_{\overline{V}_i}\,\right|\,a_i\in \R\right\}\,\mathrm{and}\,X'=\left\{\left.\sum_{i=1}^{m-1}a_if|_{\overline{V}_i}\,\right|\,a_i\in \R\right\}.
\end{equation*}
By the proof of Theorem \ref{thm:nodal-anti-graph}, we have
\begin{equation}
    \mathcal{R}_p^\sigma(g) \geq \lambda_{k}, \,\,\text{for any}\,\,g\in X\setminus\vec0.
\end{equation}

By Proposition \ref{index}, $\gamma(X\cap \mathcal{S}_p)=m\ge n-k-r+3$ and $\gamma(X'\cap \mathcal{S}_p)=m-1\geq n-k-r+2$. 
Applying Lemma \ref{lemma:intersect}, $A\cap X'\ne\emptyset$ for any centrally symmetric compact subset $A\subset \mathcal{S}_p$  with  $\gamma(A)\ge k+r-1$.   Then we have $$\lambda_{k+r-1}=\inf\limits_{A\in \mathcal{F}_{k+r-1}(\mathcal{S}_p)}\sup\limits_{g'\in A}
\mathcal{R}_p^\sigma(g')\ge \inf\limits_{g\in X'\cap \mathcal{S}_p}
\mathcal{R}_p^\sigma(g)\ge\inf\limits_{g\in X\cap \mathcal{S}_p}
\mathcal{R}_p^\sigma(g) \ge \lambda_{k}=\lambda_{k+r-1}.$$
Therefore, all the inequalities above are equalities. Next, by Lemma \ref{lemma:exsit},  $X'\cap \mathcal{S}_p$ contains a critical point of $\RQ_p^\sigma$ corresponding to $\lambda_k$. That is, there exists an eigenfunction $\overline{g}=\sum_{i=1}^{m-1}b_if|_{\overline{V}_i}\in X'\setminus\vec0$ corresponding to the eigenvalue $\lambda_k$, which contradicts to the fact that $f$ has minimal support. This shows $m\leq n-k-r+2$.

In the particular case of $p=1$, we actually have $\mathfrak{S}(f)=1$ by Lemma \ref{lem:localization-1}. Moreover, we can assume without loss of generality that $f=1_A-1_B$ for disjoint subsets $A$ and $B$, where $A\cup B$ is the strong nodal domain of $f$. When the graph is balanced, we obtain by the definition of strong nodal domains that  $\overline{\mathfrak{S}}(f)=|A\cup B|$. The estimate $\overline{\mathfrak{S}}(f)\leq n-r-k+2$ proved above tells $|A\cup B|\le n-r-k+2$. Consequently, the number of zeros of $f$ is at least $k+r-2$. 
\end{proof}
Next, we present two important applications of the upper bounds for $\mathfrak{S}(f)$, $\overline{\mathfrak{S}}(f)$, $\mathfrak{W}(f)$ and $\overline{\mathfrak{W}}(f)$ in Theorem \ref{thm:nodal-signed-graph}, Theorem \ref{thm:nodal-anti-graph}, and Theorem \ref{thm:weak nodal domain}. The estimates of the quantity 
\[\mathfrak{S}(f)+\overline{\mathfrak{S}}(f)\] for an eigenfunction $f$ will play an essential role.

\begin{theorem}\label{thm:non-var}
Let $\Gamma=(G,\sigma)$ be a signed graph with $G=(V,E)$. Let $f$ be an eigenfunction corresponding to a non-variational eigenvalue. If $|E|<|V|$, then $f$ must have zeros. 
\end{theorem}

We emphasize that the graph $G=(V,E)$ in the above theorem is allowed to be disconnected. 

\begin{proof}
    We prove it by contradiction. We assume that $f$ is non-zero on all vertices. Define \[E^+:=\{\{x,y\}\in E :f(x)\sigma_{xy}f(y)>0\}\,\,\text{and}\,\, E^-:=\{\{x,y\}\in E : f(x)\sigma_{xy}f(y)<0\}.\] 
    By assumption, we have $|E|=|E^+|+|E^-|.$ By definition of strong nodal domains, we have
\begin{equation*}
    \mathfrak{S}(f)\geq n-|E^+|\,\,\text{and}\,\, \overline{\mathfrak{S}}(f)\geq n-|E^-|,
\end{equation*}
where $n=|V|$.
Let $k$ be the index such that $\lambda_k<\lambda<\lambda_{k+1}$, where $\lambda$ is the eigenvalue to $f$. Then, Theorems \ref{thm:nodal-signed-graph} and  \ref{thm:nodal-anti-graph} tell that
\begin{equation*}
    \mathfrak{S}(f)\leq k\,\,\text{and}\,\,\overline{\mathfrak{S}}(f)\leq n-k.
\end{equation*}
Combining the above inequalities, we have 
\begin{equation*}
    n\geq  \mathfrak{S}(f)+\overline{\mathfrak{S}}(f)\geq 2n-|E^+|-|E^-|=2n-|E|>n,
\end{equation*}
which is a contradiction.
\end{proof}

   On a forest $G$, Theorem \ref{thm:non-var} implies that any eigenvalue $\lambda$ with an everywhere non-zero eigenfunction $f$ must be a variational eigenvalue. This can be strengthened as follows. Theorem \ref{thm:forest} below has been obtained in \cite[Theorem 3.8]{DPT21}. We provide here an alternative simple proof using the estimates of nodal domains and anti-nodal domains.  

\begin{theorem}\label{thm:forest}
      Let $G=(V,E)$ be a forest with $c$ connected components and $f$ be an everywhere non-zero eigenfunction corresponding to an eigenvalue $\lambda$. Then $\lambda$ is a variational eigenvalue with variational multiplicity $c$ and $f$ has exactly $k+c-1$ strong nodal domains.
\end{theorem}
This can be regarded as a non-linear version of the results on the linear Laplacian \cite{Berkolaiko08, Biyikoglu03, Fiedler75}.
\begin{proof}
  Since $G$ is a forest, we have $|V|-|E|=c>0$. By Theorem \ref{thm:non-var} and the assumption that $f$ is non-zero on every vertex, $\lambda$ is a variational eigenvalue. We assume that $\lambda=\lambda_k$ and 
  \begin{equation*}
      \lambda_{k-1}<\lambda_{k}=\cdots=\lambda_{k+r-1}<\lambda_{k+r}.
  \end{equation*}
We define $E^+:=\{\{x,y\}\in E :f(x)\sigma_{xy}f(y)>0\}$ and $E^-:=\{\{x,y\}\in E :f(x)\sigma_{xy}f(y)<0\}$. Since $f$ does not have zeros, we have $|E|=|E^+|+|E^-|.$ By definition of strong nodal domains, we have  
\[\mathfrak{S}(f)=n-|E^+|\,\,\text{and}\,\, \overline{\mathfrak{S}}(f)=n-|E^-|,\]
where $n=|V|$. This yields
\begin{equation*}
    \mathfrak{S}(f)+\overline{ \mathfrak{S}}(f)=2n-|E^+|-|E^-|=n+c.
\end{equation*}
We first prove $r\leq c$. Since $f$ is non-zero on every vertex, we can use Theorem \ref{thm:weak nodal domain} to get 
\begin{equation*}
    \mathfrak{S}(f)=\mathfrak{W}(f)\leq k+c-1\,\,\text{and}\,\,\overline{\mathfrak{S}}(f)=\overline{\mathfrak{W}}(f)\leq n-k-r+c+1.
\end{equation*}
We compute
\begin{equation}\label{eq:anti-weak}
    n+c=\mathfrak{S}(f)+\overline{ \mathfrak{S}}(f)\leq k+c-1+n-k-r+c+1=n+2c-r,
\end{equation}
which implies $r\leq c$.

Next, we prove $r\geq c$. By Theorems \ref{thm:nodal-signed-graph} and  \ref{thm:nodal-anti-graph}, we have \[\mathfrak{S}(f)\leq k+r-1\,\,\text{and}\,\,\overline{\mathfrak{S}}(f)\leq n-k+1.\]
Hence, we obtain
\begin{equation}
    n+c=\mathfrak{S}(f)+\overline{ \mathfrak{S}}(f)\leq k+r-1+n-k+1=n+r.
\end{equation}
Then we have $c\leq r$. This concludes that $r=c$. So the equality holds in \eqref{eq:anti-weak}, which implies that $\mathfrak{S}(f)=k+c-1$ and $\overline{\mathfrak{S}}(f)=n-k+1.$
\end{proof}

We present below a nonlinear version of \cite[Theorem 3.18]{GL21+}. It tells that there exist many weighted signed graphs for which all the eigenfunctions of the first variational eigenvalue have no zeros.

\begin{theorem}
Let $\Gamma=(G,\sigma)$ be a connected signed graph with $G=(V,E)$. For any vertex weight $\mu$ and potential function $\kappa$, there exists an edge measure $w:E\to\mathbb{R}^+$ compatible with  $G$, that is $w_{xy}\neq 0$ if and only if $x\sim y$, such that the first variational eigenvalue $\lambda_1$ of $\Delta_p^\sigma, \,p>1$ has multiplicity $1$, and any corresponding eigenfunction $f_1$ is nonzero everywhere. 
\end{theorem}

\begin{proof}
Choose a switching function $\tau$ such that $\Gamma^\tau=(G,\sigma^\tau)$ has a spanning tree consisting of positive edges. Set $E_+^\tau:=\{\{x,y\}\in E:\sigma_{xy}^\tau=+1\}$ and $E_-^\tau:=\{\{x,y\}\in E:\sigma_{xy}^\tau=-1\}$. Then, the graph $(V,E_+^\tau)$ is a connected graph. By \cite[Theorem 4.1]{DPT21}, for any edge measure $w^+$ on $E_+^\tau$,  the first variational eigenvalue of the $p$-Laplacian on the graph $(V,E_+^\tau)$ has multiplicity $1$ and the corresponding eigenfunctions are either positive  everywhere or negative everywhere. 

For any  edge measure $w^-$ on $E_-^\tau$, and for any $\epsilon>0$, $w^++\epsilon w^-$ provides an edge measure on $E$. In fact, we have for every edge $\{x,y\}\in E$ that
\begin{equation*}
    \left(w^++\epsilon w^-\right)_{xy}=\begin{cases}w^+_{xy},&\text{ if }\{x,y\}\in E_+^\tau,\\ \epsilon w^-_{xy},&\text{ if } \{x,y\}\in E_-^\tau.\end{cases}
\end{equation*}
 By Proposition \ref{pro:continuity}, 
it still holds for $\epsilon$ sufficiently small that the first eigenvalue $\lambda_1$ of $\Delta_p^{\sigma^{\tau}}$ on the signed graph $\Gamma^{\tau}$ equipped with the edge measure $w^++\epsilon w^-$ has multiplicity $1$ and any corresponding eigenfunction $f_1$ is either positive everywhere or negative everywhere. By Proposition \ref{pro:switching-spectra},  $\tau f_1$ is an eigenfunction of $\Delta_p^{\sigma}$ on the signed graph $\Gamma$ with the edge measure $w^++\epsilon w^-$, and $\lambda_1$ is the corresponding eigenvalue of $\Delta_p^{\sigma}$. 
\end{proof}
To conclude this section, we point out that nodal domain properties of unbalanced signed graphs are quite different from that of balanced ones.  Let $p>1$. Recall from Theorem \ref{thm:weak nodal domain} that $\mathfrak{W}(f_1)=1$, where $f_1$ is the eigenfunction corresponding to the first variational eigenvalue $\lambda_1$. Let $f_2$ be an eigenfunction corresponding to the second variational eigenvalue $\lambda_2$. When the graph is connected and balanced, we have by \cite[Theorem 4.1]{DPT21} and Theorem \ref{thm:weak nodal domain} that $1<\mathfrak{W}(f)\leq 2$, and hence, $\mathfrak{W}(f)=2$. However, this is not always true for unbalanced signed graphs. Indeed, the first variational eigenvalue $\lambda_1$ of $\Delta_p^\sigma$ on an unbalanced signed graph can have high multiplicity. Therefore, it can happen that $\mathfrak{W}(f_2)=1$. The following example tells that, even if $\lambda_1$ has multiplicity $1$, it is still possible that $\mathfrak{W}(f_2)=1$.
\begin{example}\label{ex:weak}
   Let $p=2$. Consider the complete graph $K_7$ with the signature $\sigma\equiv-1$. Define the symmetric matrix $A$ where $A_{ij}=1$ for any $i\neq j$ and $A_{ii}=i$ for $i=1,\ldots,7$. By construction, $A$ is compatible with $(K_7,\sigma)$. It is direct to check that every eigenvalue has multiplicity $1$ and the number of weak nodal domains of the second eigenfunction is $1$.
\end{example}

\section{Cheeger inequalities related to nodal domains}
\label{sec:Cheeger}

In this section, we assume the potential function $\kappa=0$. Let us first introduce the following notations. For any subsets $V_1,V_2\subset V$, we denote by 
\begin{equation*}
    |E^{\pm }(V_1,V_2)|:=\sum_{x\in V_1}\sum_{\substack{y\in V_2\\ \sigma_{xy}=\pm 1}}w_{xy}.
\end{equation*}
When $V_1=V_2$, we write $|E^{\pm }(V_1)|=|E^{\pm }(V_1,V_1)|$ for short. We further have the following notations for boundary measure and volume:
\begin{equation*}
   |\partial V_1|:=\sum_{x\in V_1}\sum_{y\not\in V_1}w_{xy},\,\,\text{and}\,\,\mathrm{vol}_\mu(V_1):=\sum_{x\in V_1}\mu(x).
\end{equation*}
For ease of notation, we denote $n=|V|$.
\begin{defn}\cite[Definition 3.2]{AtayLiu}
For any integer $1\leq k\leq n$, the $k$-way signed Cheeger constant $h_k^\sigma$ of a signed graph $\Gamma=(G,\sigma)$ is defined as 
\begin{equation*}
    h_k^\sigma:=\min_{\{(V_{2i-1}, V_{2i})\}_{i=1}^k}\max_{i=1,\ldots,k}\beta^\sigma(V_{2i-1},V_{2i}),
\end{equation*}
where 
\[\beta^\sigma(V_{2i-1},V_{2i}):=\frac{2|E^+(V_{2i-1},V_{2i})|+|E^-(V_{2i-1})|+|E^-(V_{2i})|+|\partial(V_{2i-1}\cup V_{2i})|}{\mathrm{vol}_{\mu}(V_{2i-1}\cup V_{2i})},\]
and the minimum is taken over all possible $k$-sub-bipartitions, i.e., $(V_{2i-1}\cup V_{2i})\bigcap (V_{2j-1}\cup V_{2j})=\emptyset$ for any $i\neq j$, and $V_{2l-1}\cup V_{2l}\neq \emptyset$, $V_{2l-1}\cap V_{2l}=\emptyset$ for any $l$.
\end{defn}
It is direct to check the following monotonicity of the multi-way singed Cheeger constants. For the readers' convenience, we provide a proof below.
\begin{lemma}[Monotonicity]\label{lemma:monotonicity}
    For any integer $1\leq k \leq n-1$, we have $h_k^{\sigma}\leq h_{k+1}^{\sigma}$.
\end{lemma}
\begin{proof}
    Let $\{(V_{2i-1},V_{2i})\}_{i=1}^{k+1}$ be a $(k+1)$-sub-bipartitions of $V$ satisfying $h_{k+1}^{\sigma}=\max_{1\leq i\leq k+1} \beta^{\sigma}(V_{2i-1},V_{2i})$. We define a new $k$-sub-bipartitions $\{(U_{2l-1},U_{2l})\}_{l=1}^{k}$ of $V$ as follows:
 \begin{equation*}
U_m=\left\{
	\begin{aligned}
	&V_m, \qquad \qquad 1\leq m \leq  2k-2,\\
	&V_{m}\cup V_{m+2}, \quad m=2k-1 \text{ or } 2k.\\
	\end{aligned}
	\right
	.
\end{equation*}
By definition, we have $\beta^{\sigma}(U_{2l-1},U_{2l})=\beta^{\sigma}(V_{2l-1},V_{2l})$ for any $1\leq l \leq k-1$. Next, by direct computation, we get
\begin{equation*}
    \beta^{\sigma}(U_{2k-1},U_{2k})\leq \max\{\beta^{\sigma}(V_{2k-1},V_{2k}),\beta^{\sigma}(V_{2k+1},V_{2k+2})\}.
\end{equation*}
So this implies
\begin{equation*}
    h_k^\sigma=\min_{\{(W_{2i-1}, W_{2i})\}_{i=1}^k}\max_{i=1,\ldots,k}\beta^\sigma(W_{2i-1},W_{2i})\leq \max_{i=1,\ldots,k}\beta^\sigma(U_{2i-1},U_{2i})\leq \max_{1\leq i\leq k+1} \beta^{\sigma}(V_{2i-1},V_{2i}) = h_{k+1}^{\sigma}.
\end{equation*}
\end{proof}
\begin{remark}
    The above signed Cheeger constants on signed graphs can be considered as an optimization of a mixture of isoperimetric constant and the so-called frustration index. The frustration index $\iota^\sigma(\Omega)$ of a subset $\Omega\subset V$ measures how far the signature on $\Omega$ is from being balanced. It is defined as 
    \begin{equation*}
        \iota^\sigma(\Omega):=\min_{\tau: \Omega\to \{\pm 1\}}\sum_{\substack{x,y\in \Omega\\x\sim y}}|\tau(x)-\sigma_{xy}\tau(y)|.
    \end{equation*}
By switching, we see $\iota^\sigma(\Omega)=0$ if and only if the signature restricting to the subgraph induced by $\Omega$ is balanced. Indeed, the $k$-th signed Cheeger inequality can be reformulated as \cite{LLPP15}
\begin{equation*}
      h_k^\sigma:=\min_{\{\Omega_i\}_{i=1}^k}\max_{i=1,\ldots,k}\frac{\iota^\sigma(\Omega_i)+|\partial \Omega_i|}{\mathrm{vol}_{\mu}(\Omega_i)}.
\end{equation*}
This can be verified using the one-to-one correspondence between the function $\tau:\Omega_i\to \{\pm 1\}$ and the bipartition $(V_{2i-1}, V_{2i})$ of $\Omega_i$ via the relation
\[V_{2i-1}:=\{x\in \Omega_i\left|\,\tau(x)=+1\}\right.,\,\,\text{and}\,\,V_{2i}:=\{x\in \Omega_i\left|\, \tau(x)=-1\}\right..\]
Notice that $h_k^\sigma$ reduces to the classical $k$-th Cheeger constant when $\Gamma=(G,\sigma)$ is balanced, since $\iota^\sigma(\Omega_i)$ vanishes for any subset $\Omega_i$.
\end{remark}

\begin{theorem}\label{thm:p-lap-C}
For any $p\ge 1$ and any $k\in\{1,\ldots, n\}$, the $k$-th variational eigenvalue  $\lambda_k(\Delta_p^\sigma)$ satisfies
$$\frac{2^{p-1}}{C^{p-1}p^p} (h_{m}^\sigma)^p\le \lambda_k(\Delta_p^\sigma)\le 2^{p-1} h_{k}^\sigma,$$
where $C:=\max_{x\in V}\frac{\sum_y w_{xy}}{\mu_x}$ and $m$ is the number of strong nodal domains of an  eigenfunction   corresponding to $\lambda_k(\Delta_p^\sigma)$.
\end{theorem}
This theorem can be regarded as a signed version of \cite[Theorem 5.1]{TudiscoHein18}, which is an extension of previous works \cite{Amghibech,Chang, HeinBuhler2009,DHJ10}. 

Before proving this theorem, we first show an elementary inequality.
\begin{lemma}\label{lemma:convex}
For any $a,b\in \mathbb{R}$, $p\geq 1$ and $\sigma_{ab}\in \{-1,1\}$, we have
\[|a-\sigma_{ab} b|^p\le 2^{p-1}\Big||a|^p\mathrm{sgn}(a)-\sigma_{ab} |b|^p\mathrm{sgn}(b)\Big|.\] 
\end{lemma}
\begin{proof}
Without loss of generality, we can assume $ab\neq 0$.
We consider the case of $\sigma_{ab}=-1$ below. The proof for the case of $\sigma_{ab}=1$ can be done similarly. 

If $ab>0$, we assume $a>0$ and $b>0$ without loss of generality. Then we get
\begin{equation*}
     |a-\sigma_{ab}b|^p =|a+b|^p.
\end{equation*}
By the convexity of  $f(x)=|x|^p$, we have $f(\frac{a+b}{2})\le \frac{1}{2}f(a)+\frac{1}{2}f(b)$, i.e., 
\begin{align*}
     |a+b|^p & \le 2^{p-1}\left(|a|^p+ |b|^p\right) \\
     &= 2^{p-1}\Big||a|^p\mathrm{sgn}(a)-\sigma_{ab} |b|^p\mathrm{sgn}(b)\Big|.
\end{align*}

If, otherwise, $ab<0$, we assume $a>0$, $b<0$, and $a=-kb$ with $k>1$ without loss of generality. Then we get
\begin{equation*}
     |a-\sigma_{ab} b|^p =|a+b|^p=|k-1|^p|b|^p.
\end{equation*}
By the convexity of the following function
\begin{equation*}
g(x)=\begin{cases}|x|^p,&\text{ if }x\geq0,\\ x,&\text{ if } x<0,\end{cases}
\end{equation*}
we have $g(\frac{k-1}{2})\leq \frac{1}{2}g(k)+\frac{1}{2}g(-1)$, i.e., $|k-1|^p\leq 2^{p-1}\left(|k|^p-1\right)$. Next, we compute
\begin{align*}
   |k-1|^p|b|^p & \le 2^{p-1}(|k^p|-1)|b|^p\\
     &= 2^{p-1}\Big||a|^p\mathrm{sgn}(a)-\sigma_{ab} |b|^p\mathrm{sgn}(b)\Big|.
\end{align*}
This completes the proof of the case $\sigma_{ab}=-1$.
\end{proof}
\begin{proof}[Proof of Theorem \ref{thm:p-lap-C}]
Observe that for any  $k$-sub-bipartitions $\{(V_{2i-1},V_{2i})\}_{i=1}^k$ of $V$ that 
\begin{equation*}
    \RQ_1^\sigma(1_{V_{2i-1}}-1_{V_{2i}})=\beta^\sigma(V_{2i-1},V_{2i}),
\end{equation*}
 where $1_{V_i}$ is the  indicator function of $V_i$. 

We first show the upper bound estimate of $\lambda_k$. By abuse of notation, we use $\{(V_{2i-1},V_{2i})\}_{i=1}^k$ for a $k$-sub-bipartitions  of $V$ that  realizes  $h_k^\sigma$, i.e., $h_k^\sigma=\max\limits_{1\le i\le k}\beta^\sigma(V_{2i-1},V_{2i})$. For any $g\in \mathrm{span}(1_{V_{1}}-1_{V_{2}},\cdots,1_{V_{2k-1}}-1_{V_{2k}})$, i.e., \[g(x)=\sum_{i=1}^kt_i(1_{V_{2i-1}}(x)-1_{V_{2i}}(x))\,\,\text{with}\,\, t_1,\cdots,t_k\in\R,\] we derive by Lemma \ref{lemma:convex} that
\begin{align*}
|g(x)-\sigma_{xy}g(y)|^p&\le 2^{p-1}\left||g(x)|^p\mathrm{sgn}(g(x))-\sigma_{xy}|g(y)|^p\mathrm{sgn}(g(y))\right| 
\\&=2^{p-1}\left| \sum_{i=1}^{k}|t_i|^p\mathrm{sgn}(t_i)\left(1_{V_{2i-1}}(x)-1_{V_{2i}}(x)-\sigma_{xy}(1_{V_{2i-1}}(y)-1_{V_{2i}}(y))\right)\right|   
\\&\le 2^{p-1} \sum_{i=1}^{k}|t_i|^p\left|1_{V_{2i-1}}(x)-1_{V_{2i}}(x)-\sigma_{xy}(1_{V_{2i-1}}(y)-1_{V_{2i}}(y))\right|.
\end{align*}
Therefore, we compute
\begin{align*}
\RQ_p^{\sigma}(g)&=\frac{\sum_{\{x,y\}\in E}w_{xy}|g(x)-\sigma_{xy} g(y)|^p}{\sum_{x\in V}\mu_x|g(x)|^p}
\\&\leq 2^{p-1}\frac{\sum_{\{x,y\}\in E}w_{xy}\sum_{i=1}^{k}|t_i|^p\left|1_{V_{2i-1}}(x)-1_{V_{2i}}(x)-\sigma_{xy}(1_{V_{2i-1}}(y)-1_{V_{2i}}(y))\right|}{\sum_{i=1}^{k}\sum_{x\in V_{2i-1}\cup V_{2i} }\mu_x|t_i|^p|1_{V_{2i-1}}(x)-1_{V_{2i}}(x)|}
\\&= 2^{p-1}\frac{\sum_{i=1}^{k}|t_i|^p\sum_{\{x,y\}\in E}w_{xy}\left|1_{V_{2i-1}}(x)-1_{V_{2i}}(x)-\sigma_{xy}(1_{V_{2i-1}}(y)-1_{V_{2i}}(y))\right| }{\sum_{i=1}^{k}|t_i|^p\sum_{x\in V}\mu_x|1_{V_{2i-1}}(x)-1_{V_{2i}}(x)| }
\\&\le 2^{p-1} \max\limits_{i=1,\cdots,k} \RQ_1(1_{V_{2i-1}}(x)-1_{V_{2i}}(x))
\\&=2^{p-1}\max\limits_{1\le i\le k}\beta^\sigma(V_{2i-1},V_{2i})\\&=2^{p-1}h_k^\sigma.
\end{align*}
By definition of the variational eigenvalue  $\lambda_k$, we obtain $\lambda_k\le 2^{p-1}h_k^\sigma$.

Next, we prove the lower bound estimate of $\lambda_k$.
Let $f$  be an eigenfunction corresponding to   $\lambda_k$, and let $V_1,\cdots,V_m$ be the strong nodal domains of $f$. By the proof of Theorem \ref{thm:nodal-signed-graph}, we have \[\RQ_p^\sigma(f_i)\le \lambda_k, \,\,i=1,\ldots,m,\] where $f_i$ equals  $f$ on $V_i$ and equals zero otherwise.

 We prove two claims.

\textbf{Claim 1.} For any $i=1,\ldots,m$, we denote by $f_i^p: V\to \mathbb{R}$ the function $x\mapsto |f_i(x)|^p\mathrm{sgn}(f_i(x))$. Then we have
\[\RQ_p^\sigma(f_i)\ge \frac{2^{p-1}}{C^{p-1}p^p} \RQ_1^\sigma(f_i^p)^p,\,\,i=1,\cdots,m.\]

 Indeed, by \cite[Lemma 3]{Amghibech}, we have
 \begin{align*}
     \left|f_i^p(x)-\sigma_{xy}f_i^p(y)\right|\leq p|f_i(x)-\sigma_{xy}f_i(y)|\left(\frac{|f_i^p(x)|+|f_{i}^p(y)|}{2}\right)^{1-\frac 1p}.
 \end{align*}
Following the proof of \cite[Lemma 5.2]{TudiscoHein18}, we obtain 
\begin{align*}
\RQ_1^\sigma(f_{i}^p)&=\frac{\sum_{\{x,y\}\in E}w_{xy}|f_{i}^p(x)-\sigma_{xy}f_{i}^p(y)|
}{\sum_{x\in V_i}\mu_x|f(x)|^p} 
\\&\leq p\frac{\sum_{\{x,y\}\in E}w_{xy}|f_i(x)-\sigma_{xy}f_i(y)|\left(\frac{|f_i^p(x)|+|f_{i}^p(y)|}{2}\right)^{1-\frac 1p}
}{\sum_{x\in V_i}\mu_x|f(x)|^p} 
\\&\leq p\frac{\left(\sum_{\{x,y\}\in E}w_{xy}|f_i(x)-\sigma_{xy}f_i(y)|^p\right)^{\frac{1}{p}}\left(\sum_{\{x,y\}\in E}w_{xy}\frac{|f_i^p(x)|+|f_{i}^p(y)|}{2}\right)^{1-\frac{1}{p}}
}{\sum_{x\in V_i}\mu_x|f(x)|^p} 
\\&\le p\left(\frac{\sum_{\{x,y\}\in E}w_{xy}|f_{i}(x)-\sigma_{xy}f_{i}(y)|^p
}{\sum_{x\in V_i}\mu_x|f(x)|^p}\right)^{\frac1p}\left(\frac C2\right)^{1-\frac1p}
\\&=p\left(\RQ^\sigma_p(f_{i})\right)^{\frac1p}\left(\frac C2\right)^{1-\frac1p}.
\end{align*}
This proves Claim 1.

\textbf{Claim 2.} There exist $U_{2i-1}\sqcup U_{2i}\subset V_{i}$, $i=1,\cdots,m$, such that
\[\RQ_1^\sigma(f_{i}^p)\ge \RQ_1^\sigma(1_{U_{2i-1}}-1_{U_{2i}})=\beta^\sigma(U_{2i-1},U_{2i}).\]

For any $t\geq 0$, define $V_\pm^t(f):=\left\{x\in V:\pm f(x)>t^{\frac1p}\right\}$ and a function $\hat{f}^t: V\to \mathbb{R}$ as follows
\begin{equation*}
\hat{f}^t(x)=\begin{cases}1,&\text{if }f(x)>t^{\frac{1}{p}},\\ -1,&\text{if } f(x)<-t^{\frac{1}{p}},\\
0, &\text{otherwise.}\end{cases}
\end{equation*}
Then, we have
$$\int_0^\infty\sum_{x\in V_i}\mu_x|\hat{f}^t(x)|^pdt=\sum_{x\in V_i}\mu_x\int_0^\infty|\hat{f}^t(x)|^pdt=\sum_{x\in V_i}\mu_x\int_0^{|f(x)|^p}1dt=\sum_{x\in V_i}\mu_x|f(x)|^p,$$
and 
\begin{align*}
\int_0^{\infty}\left|\hat{f}^t_i(x)-\sigma_{xy}\hat{f}^t_i(y)\right|dt=|f_i^p(x)-\sigma_{xy}f_i^p(y)|,\, 
\text{\,for any\,} \{x,y\}\in E.
\end{align*}
Note that the function $f_i^p$ is defined as in Claim 1.
So by direct calculation, we have
\begin{align*}
    \int_0^\infty \sum_{\{x,y\}\in E}w_{xy}\left|\hat{f}^t_{i}(x)-\sigma_{xy}\hat{f}^t_{i}(y)\right|dt&=\sum_{\{xy\}\in E}w_{xy}\int_0^\infty \left|\hat{f^t_{i}}(x)-\sigma_{xy}\hat{f^t_{i}}(y)\right|dt
    \\&= \sum_{\{x,y\}\in E}w_{xy}\left|f_{i}^p(x)-\sigma_{xy}f_{i}^p(y)\right|.
\end{align*}
Therefore, there exists $t_0\ge0$ such that 
\begin{align*}
\RQ_1^\sigma(f_{i}^p)&=\frac{\sum_{\{x,y\}\in E}w_{xy}|f_{i}^p(x)-\sigma_{xy}f_{i}^p(y)|}{\sum_{x\in V}\mu_x|f_{i}(x)|^p}
\\&=\frac{\int_0^\infty \sum_{\{x,y\}\in E}w_{xy}\left|\hat{f}^t_i(x)-\sigma_{xy}\hat{f}^t_i(y)\right|dt}{\int_0^\infty\sum_{x\in V_i}\mu_x|\hat{f}^t(x)|^pdt}
\\&\ge \frac{\sum_{\{x,y\}\in E}w_{xy}\left|\hat{f}^{t_0}_{i}(x)-\sigma_{xy}\hat{f}^{t_0}_{i}(y)\right|}{\sum_{x\in V}\mu_x|\hat{f}^{t_0}_{i}(x)|^p}
\\&=\RQ_1^\sigma(\hat{f}^{t_0}_{i})=
\RQ_1^\sigma(1_{U_{2i-1}}-1_{U_{2i}})=\beta^\sigma(U_{2i-1},U_{2i}),
\end{align*}
where $U_{2i-1}:=V_+^{t_0}(f_{i})$ and  $U_{2i}:=V_-^{t_0}(f_{i})$. This completes the proof of Claim 2. 

Combining the above two claims, we  get \[\frac{2^{p-1}}{C^{p-1}p^p}(\beta^\sigma(U_{2i-1},U_{2i}))^p\le \lambda_k,\,\,\text{for}\,\, i=1,\cdots,m.\]  In consequence, $\frac{2^{p-1}}{C^{p-1}p^p} (h_{m}^\sigma)^p\le \lambda_k$. The proof is completed.
\end{proof}


It was proved in \cite[Proposition 3.2]{AtayLiu} that  $h_1^\sigma=\cdots=h_k^\sigma=0<h_{k+1}^\sigma$  if and only if $\Gamma=(G,\sigma)$ has exactly $k$  balanced connected components. Combining Theorem \ref{thm:p-lap-C} with \cite[Proposition 3.2]{AtayLiu}, we derive the proposition below.  

\begin{pro}\label{pro:4.1}
For any $p\geq 1$ and any $k\in\{0,1,\ldots,n\}$, a signed graph $\Gamma$ has exactly $k$ balanced connected components if and only if the variational eigenvalues of the $p$-Laplacian satisfy  \[\lambda_1(\Delta_p^{\sigma})=\cdots=\lambda_k(\Delta_p^{\sigma})=0<\lambda_{k+1}(\Delta_p^{\sigma}).\]
 Moreover, suppose $\Gamma$ has $k+l$ connected components denoted by  $\Gamma_1,\ldots,\Gamma_{k+l}$, in which   $\Gamma_1,\ldots,\Gamma_{k}$ are balanced,  while  $\Gamma_{k+1},\ldots,\Gamma_{k+l}$ are not  balanced. Then, the smallest positive eigenvalue of the $p$--Laplacian coincides with the $(k+1)$-th variational eigenvalue, which can be expressed as follows
 \begin{equation}\label{eq:k+1-variational}
  \lambda_{k+1}(\Delta_p^{\sigma})=\min\left\{\min\limits_{i=1,\cdots,k}\lambda_{2}(\Delta_p^{\sigma}|_{\Gamma_i}),\min\limits_{j=k+1,\cdots,k+l}\lambda_{1}(\Delta_p^{\sigma}|_{\Gamma_j})\right\},
 \end{equation}
  where $\lambda_{s}(\Delta_p^{\sigma}|_{\Gamma_i})$ indicates the $s$-th variational eigenvalue of the $p$-Laplacian restricted on $\Gamma_i$.
\end{pro}

\begin{proof}
We first assume $\Gamma$ has exactly $k$  balanced connected components. Then by \cite[Proposition 3.2]{AtayLiu}, we have $h_1^\sigma=\cdots=h_k^\sigma=0<h_{k+1}^\sigma$. By Theorem \ref{thm:p-lap-C}, we have $\lambda_k(\Delta_p^{\sigma})\le 2^{p-1}h_k^\sigma=0$. Since $0\le\lambda_1(\Delta_p^{\sigma})\le\cdots\le\lambda_k(\Delta_p^{\sigma})$, they are all zero.
On the other hand, according to \cite[Theorem 2.1]{JZ21-PM}, the smallest positive eigenvalue of $\Delta_p^\sigma$ on $\Gamma$ is $\lambda_{k+1}(\Delta_p^{\sigma})$. So we have $\lambda_{k+1}(\Delta_p^{\sigma})>0$. Conversely, we assume that $\lambda_1(\Delta_p^{\sigma})=\cdots=\lambda_k(\Delta_p^{\sigma})=0<\lambda_{k+1}(\Delta_p^{\sigma})$. Denote by $m$ the number of balanced connected components of $\Gamma$. Along the same line of the above arguments, we derive that \[\lambda_1(\Delta_p^{\sigma})=\cdots=\lambda_m(\Delta_p^{\sigma})=0<\lambda_{m+1}(\Delta_p^{\sigma}).\]
Comparing with our assumption, we have $m=k$.

Next, we prove \eqref{eq:k+1-variational}.
It is direct to check that the eigenvalue of $\Delta_p^\sigma$ on $\Gamma$ is the multiset-sum of  the  eigenvalue of  $\Delta_p^\sigma$ on $\Gamma_i$ for $i=1,\cdots,k+l$, i.e., 
\[\{ \lambda: \lambda \text{ is an eigenvalue of $\Delta_p^\sigma$ on $\Gamma$}\}=\oplus_{i=1}^{k+l}\{\lambda:\lambda\text{ is an eigenvalue of $\Delta_p^\sigma$ on $\Gamma_i$}\}.\] 
Therefore, the smallest positive eigenvalue of $\Delta_1^\sigma$ on $\Gamma$ coincides with 
 $$\min\limits_{i=1,\cdots,k+l}\,\{\text{the smallest positive eigenvalue of } \Delta_p^\sigma\text{ on } \Gamma_i\}.$$
 Noticing the result from \cite[Theorem 2.1]{JZ21-PM}, this completes the proof of \eqref{eq:k+1-variational}.
\end{proof}

For particular cases, the variational eigenvalues of the $1$-Laplacian might coincide with the signed Cheeger constants.
\begin{cor}\label{cor:cheeger}
   For any signed graph $\Gamma=(G,\sigma)$, we have $\lambda_1(\Delta_1^{\sigma})=h_1^{\sigma}$. Moreover, if $\Gamma$ is balanced, we have $\lambda_2(\Delta_1^{\sigma})=h_2^{\sigma}$.
\end{cor}
\begin{proof}
    Let $f_1$ be an eigenfunction corresponding to $\lambda_1(\Delta_1^{\sigma})$. Setting $p=1$ and $k=1$ in Theorem \ref{thm:p-lap-C} leads to $h_{m}^{\sigma}\leq \lambda_1(\Delta_1^{\sigma})\leq h_1^{\sigma}$, where $m=\mathfrak{S}(f_1)$. Since $m\geq 1$, we have $h_1^{\sigma}\leq h_m^{\sigma}$. This implies $\lambda_1(\Delta_1^{\sigma})=h_1^{\sigma}$.
    When $\Gamma$ is balanced,
    the identity $\lambda_2(\Delta_1^{\sigma})=h_2^{\sigma}$ follows directly from Proposition \ref{pro:switching-spectra} and \cite[Theorem 5.15]{Chang}.
\end{proof}

As a consequence of Proposition \ref{pro:4.1} and Corollary \ref{cor:cheeger}, we have the following expression of the first positive eigenvalue of the $1$-Laplacian. 

\begin{pro}\label{pro:first-positive}
Suppose a signed graph $\Gamma$ has $k+l$ connected components denoted by  $\Gamma_1,\ldots,\Gamma_{k+l}$, in which   $\Gamma_1,\ldots,\Gamma_{k}$ are balanced,  while  $\Gamma_{k+1},\ldots,\Gamma_{k+l}$ are not  balanced. Then, the smallest positive eigenvalue of the 1-Laplacian $\lambda_{k+1}(\Delta_1^\sigma)$ can be expressed via signed Cheeger constants as follows
 \begin{equation}\label{eq:cheeger}
     \lambda_{k+1}(\Delta_1^{\sigma})=h_{k+1}^\sigma(\Gamma)=\min\left\{\min\limits_{i=1,\cdots,k}h_2^\sigma(\Gamma_i),\min\limits_{j=k+1,\cdots,k+l}h_1^\sigma(\Gamma_j)\right\}.
 \end{equation}
\end{pro}

\begin{proof}
Combining Corollary \ref{cor:cheeger} with \eqref{eq:k+1-variational}, we have
\begin{equation}\label{eq:11}
\lambda_{k+1}(\Delta_1^\sigma)=\min\left\{\min\limits_{i=1,\cdots,k}h_2^\sigma(\Gamma_i),\min\limits_{j=k+1,\cdots,k+l}h_1^\sigma(\Gamma_j)\right\}.    
\end{equation}
Since the smallest positive eigenvalue of $\Delta_1^\sigma$ on $\Gamma$ is $\lambda_{k+1}(\Delta_1^\sigma)$ and  it is direct to check by definition that the quantity \eqref{eq:11} agrees with $h_{k+1}^\sigma(\Gamma)$, the proof of  \eqref{eq:cheeger} is completed.
\end{proof}

By the above results, the smallest positive eigenvalue of the $1$-Laplacian must be some multi-way signed Cheeger constant in Atay-Liu's sense. However, their multiplicities may not coincide. We show an example below.

\begin{example}\label{exam:K5}
Consider the complete graph $K_5$ with $\sigma\equiv+1$. It is direct to check that $h_1^\sigma=0$, $h_2^\sigma=\frac34$ and  $h_3^\sigma=h_4^\sigma=h_5^\sigma=1$. 
Furthermore, by the calculations in  \cite[Section 6.3]{Chang} and \cite[Proposition 4.1]{Zhang}, we have $\lambda_1(\Delta_1^\sigma)=0$, $\lambda_2(\Delta_1^\sigma)=\lambda_3(\Delta_1^\sigma)=\frac34$, and  $\lambda_4(\Delta_1^\sigma)=\lambda_5(\Delta_1^\sigma)=1$. Thus, the multiplicity of the smallest positive  eigenvalue doesn't agree with the  multiplicity of  the multi-way signed Cheeger constant $h_2^\sigma$.
\end{example}

\begin{cor}
If $\mathfrak{S}(f)=k$ for some eigenfunction  $f$ corresponding to $\lambda_1(\Delta_1^\sigma)$ or $\lambda_2(\Delta_1^\sigma)$ or the smallest positive eigenvalue of $\Delta_1^\sigma$, then we have $\lambda_i(\Delta_1^\sigma)=h_i^\sigma$,  $i=1,\cdots,k$. 
\end{cor}

\begin{proof}
We need the following simple observation: If $h_j^\sigma\le \lambda_i(\Delta_1^\sigma)$ for some $j\ge i$, then 
\begin{equation}\label{eq:observation}
\lambda_i(\Delta_1^\sigma)=\lambda_{i+1}(\Delta_1^\sigma)=\cdots=\lambda_j(\Delta_1^\sigma)= h_i^\sigma=h_{i+1}^\sigma=\cdots=h_j^\sigma.
\end{equation}
In fact, Theorem \ref{thm:p-lap-C} tells that $\lambda_j(\Delta_1^\sigma)\le h_j^\sigma$ and $ \lambda_i(\Delta_1^\sigma)\le h_i^\sigma$. Since $j\ge i$, we have $h_i^\sigma\le h_j^\sigma$ and $ \lambda_i(\Delta_1^\sigma)\le \lambda_j(\Delta_1^\sigma)$. Together with the assumption $h_j^\sigma\le \lambda_i(\Delta_1^\sigma)$, we obtain that 
\[h_j^\sigma\leq \lambda_i(\Delta_1^\sigma)\leq  \lambda_j(\Delta_1^\sigma)\leq h_j^\sigma,\,\,\text{and}\,\,h_j^\sigma\leq \lambda_i(\Delta_1^\sigma)\leq  h_i^\sigma\leq h_j^\sigma,\]
which implies immediately that $\lambda_i(\Delta_1^\sigma)= \lambda_j(\Delta_1^\sigma)=h_i^\sigma= h_j^\sigma$ and hence (\ref{eq:observation}).

Now, we move on to the  proof of the corollary. Recall from Corollary \ref{cor:cheeger} that the identity $\lambda_1(\Delta_1^\sigma)=h_1^\sigma$ always holds. So, it remains to show the case that $k\geq 2$.

If  $f$ is  an  eigenfunction  corresponding to $\lambda_1(\Delta_1^\sigma)$, then Theorem \ref{thm:p-lap-C} yields  $h_k^\sigma\le\lambda_1(\Delta_1^\sigma)$,  and the above observation implies  that $\lambda_1(\Delta_1^\sigma)=\cdots=\lambda_k(\Delta_1^\sigma)= h_1^\sigma=\cdots=h_k^\sigma$. 

If  $f$ is  an  eigenfunction  corresponding to $\lambda_2(\Delta_1^\sigma)$, then we similarly have $\lambda_2(\Delta_1^\sigma)=\cdots=\lambda_k(\Delta_1^\sigma)= h_2^\sigma=\cdots=h_k^\sigma$. The case $\lambda_1(\Delta_1^\sigma)=h_1^\sigma$ holds universally.

Suppose  $\lambda_s(\Delta_1^\sigma)$ is the smallest positive eigenvalue, and $f$ is an eigenfunction corresponding to $\lambda_s(\Delta_1^\sigma)$ with $\mathfrak{S}(f)=k$.  Without loss of generality,  we further assume that $\lambda_{s-1}(\Delta_1^\sigma)=0$.  By Proposition \ref{pro:first-positive}, we have $\lambda_{1}(\Delta_1^\sigma)=\cdots=\lambda_{s-1}(\Delta_1^\sigma)=h_1^\sigma=\cdots=h_{s-1}^\sigma=0$.  When $k\le s-1$,    nothing needs to be  proved. 
Suppose $k\ge s$. Then we apply Theorem \ref{thm:p-lap-C} and the above observation to derive $\lambda_s(\Delta_1^\sigma)=\cdots=\lambda_k(\Delta_1^\sigma)= h_s^\sigma=\cdots=h_k^\sigma>0$.
\end{proof}

\section{Perron-Frobenius theorem on antibalanced graphs}
\label{sec:Perron}
As is well known, the Perron-Frobenius theorem implies for any connected graph that, the first eigenvalue of its linear Laplacian (i.e., $p$-Laplacian with $p=2$) is simple and the corresponding eigenfunction can be taken to be positive on every vertex. For the $p$-Laplacian on graphs with $p>1$, the same property has been shown in \cite[Theorem 4.1]{DPT21} and \cite[Theorem 1.1]{HW20}. For the case of  connected antibalanced signed graphs, it was shown in \cite[Theorem 3.13]{GL21+} that the largest eigenvalue of $\Delta_p^\sigma$ with $p=2$ is simple and the corresponding eigenfunction can be taken to be positive  on every vertex. This can be considered as a Perron-Frobenius theorem for Laplacians on antibalanced signed graphs.
In this section, we prove a nonlinear version of \cite[Theorem 3.13]{GL21+} for $p$-Laplacians on antibalanced signed graphs with $p>1$ by using the estimate in Theorem \ref{thm:weak nodal domain} for the number $\overline{\mathfrak{W}}(f)$ of anti-weak nodal domains.

\begin{theorem}\label{thm:pr}
 Assume that $p>1$. Let $\Gamma=(G,\sigma)$  be a connected signed graph where  $\sigma\equiv -1$, $G=(V,E)$ and $|V|=n$. For any eigenfunction $f$ corresponding to the $n$-th variational eigenvalue $\lambda_n$ of $\Delta_p^\sigma$, we have the following properties:
	\begin{itemize}
          \item [(i)] $f$ is either strictly positive or strictly negative, i.e., either $f(x)>0$ for any $x\in V$ or $f(x)<0$ for any $x\in V$;
          \item [(ii)] For any other eigenfunction $g$ corresponding to $\lambda_n$, there exists a constant $c\in \mathbb{R}\setminus \{0\}$ such that $g=cf$;	   
		\item [(iii)]If $g$ is an eigenfunction corresponding to an eigenvalue $\lambda$, and $g(x)>0$ for any $x\in V$ or $g(x)<0$ for any $x\in V$, then $\lambda=\lambda_{n}$.
	\end{itemize}
\end{theorem}
Let us remark that the Perron-Frobenius theorem above does not hold for the case of $p=1$. Indeed, according to Theorem \ref{thm:min nodal domain}, there exists an eigenfunction $f$ corresponding to $\lambda_n$ of $\Delta_1^\sigma$ such that $\mathfrak{S}(f)=1$. However, if Theorem \ref{thm:pr} were true for $p=1$, we would have $\mathfrak{S}(f)=n$ for any eigenfunction corresponding to $\lambda_n$ of $\Delta_1^\sigma$, which is a contradiction.

\begin{proof}[Proof of Theorem \ref{thm:pr}]
(i) Since $\lambda_{n}$ is the $n$-th variational eigenvalue, Theorem \ref{thm:weak nodal domain} implies $\overline{\mathfrak{W}}(f)\leq 1$. By definition of weak nodal domains, we have  $f(x)\geq 0$ for any $x\in V$ or  $f(x)\leq 0$ for any $x\in V$. We can assume $f(x)\geq 0$ for any $x\in V$, since otherwise, we can consider the eigenfunction $-f$. 

If $f(x)=0$ for some $x\in V$, we have by the eigen-equation that
\begin{equation*}
	\Delta_p^{\sigma}f(x)=\sum_{y\sim x}w_{xy}\Phi_p\left(f(x)-\sigma_{xy}f(y)\right)+\kappa_x\Phi_p(f(x))=\lambda_n\mu_{x}\Phi_p(f(x))=0.
\end{equation*} 
Since $\sigma\equiv-1$,  we obtain $\sum_{y\sim x}w_{xy}\Phi_p(f(y))=0$. Because $f(y)$ is non-negative for all $y\in V$, we have $f(y)=0$ for all $y$ with $y\sim x$. By the connectedness of $G$, we have $f\equiv 0$. This contradicts to the assumption that $f$ is an eigenfunction of $\lambda_n$. Thus, we get $f(x)>0$ for any $x\in V$. 
\\

\vspace{0.05cm}

(ii) Suppose that $g$ is   an eigenfunction corresponding to $\lambda_{n}$.  Without loss of generality, we can assume $g(x)>0$ for any $x\in V$. By definition, we have for any $x\in V$,
\begin{equation}\label{eq:001}
	\sum_{y\sim x}w_{xy}\Phi_p(f(x)+f(y))=(\lambda_n\mu_{x}-\kappa_x)\Phi_p(f(x)),
\end{equation} 
\begin{equation}\label{eq:002}
		\sum_{y\sim x}w_{xy}\Phi_p(g(x)+g(y))=(\lambda_n\mu_{x}-\kappa_x)\Phi_p(g(x)).
\end{equation}
Multiplying  \eqref{eq:001}  by $f(x)-\frac{|g(x)|^p}{\Phi_p(f(x))}$, and \eqref{eq:002}  by $g(x)-\frac{|f(x)|^p}{\Phi_p(g(x))}$, we  derive
\begin{equation}\label{eq:003}
		\sum_{y\sim x}w_{xy}\Phi_p(f(x)+f(y))\left(f(x)-\frac{|g(x)|^p}{\Phi_p(f(x))}\right)=(\lambda_n\mu_{x}-\kappa_x)\left(|f(x)|^p-|g(x)|^p\right),
\end{equation}
\begin{equation}\label{eq:004}
	\sum_{y\sim x}w_{xy}\Phi_p(g(x)+g(y))\left(g(x)-\frac{|f(x)|^p}{\Phi_p(g(x))}\right)=(\lambda_n\mu_{x}-\kappa_x)\left(|g(x)|^p-|f(x)|^p\right).
\end{equation}
Summing  \eqref{eq:003} and \eqref{eq:004} over all $x\in V$, we get
\begin{equation}\label{eq:Rfg}
	R(f,g)+R(g,f)=0,
\end{equation}
where 
\begin{equation*}
	R(f,g)=\sum_{\{x,y\}\in E}w_{xy}\left(|g(x)+g(y)|^p-\Phi_p(f(x)+f(y))\left(\frac{|g(x)|^p}{\Phi_p(f(x))}+\frac{|g(y)|^p}{\Phi_p(f(y))}\right)\right).
\end{equation*}
We apply Lemma \ref{lemma:elementary} by setting $a=f(x)$, $b=f(y)$, $ta=g(x)$ and $sb=g(y)$ to derive that each summand in $R(f,g)$ is non-positive. Similarly, we have each summand in $R(g,f)$ is also non-positive. Therefore, the identity \eqref{eq:Rfg}  implies that every summand of $R(f,g)$ and $R(g,f)$ equals zero. By the equality condition \eqref{eq:lemma1=}  in Lemma \ref{lemma:elementary}, we have for any $\{x,y\}\in E$ that
\begin{equation*}
	\frac{g(x)}{g(y)}=\frac{f(x)}{f(y)}.
\end{equation*}
Since $G$ is connected, we drive that $g$ is proportional to $f$. This concludes the proof of (ii).\\

\vspace{0.05cm}

 (iii) If $g$ is an eigenfunction corresponding to $\lambda$ and $g(x)>0$ for any $x\in V$. By definition, we have 
\begin{equation}\label{eq:005}
	\sum_{y\sim x}w_{xy}\Phi_p(f(x)+f(y))=(\lambda_n\mu_{x}-\kappa_x)\Phi_p(f(x)),
\end{equation} 
\begin{equation}\label{eq:006}
	\sum_{y\sim x}w_{xy}\Phi_p(g(x)+g(y))=(\lambda\mu_{x}-\kappa_x)\Phi_p(g(x)).
\end{equation}
As above, we multiply \eqref{eq:005} by $f(x)-\frac{|g(x)|^p}{\Phi_p(f(x))}$ and \eqref{eq:006} by $g(x)-\frac{|f(x)|^p}{\Phi_p(g(x))}$, and sum them over all $x\in V$. Then, we obtain
\begin{equation}\label{eq:007}
      R(f,g)+R(g,f)=(\lambda_n-\lambda)\sum_{x\in V}\mu_{x}\left(|f(x)|^p-|g(x)|^p\right).
\end{equation}
We can choose sufficiently small $\epsilon>0$    such that $f(x)-\epsilon g(x)>0$ for any $x\in V$. So without loss of generality, we can assume $|f(x)|^p-|g(x)|^p>0$ for any $x\in V$. If $\lambda<\lambda_n$, then the right hand side of \eqref{eq:007} is strictly positive and the left hand side  of \eqref{eq:007} is non-positive. This is a contradiction. The proof of  $\lambda=\lambda_n$ is then completed.
\end{proof}
Notice that a connected bipartite graph with $\sigma\equiv 1$ is both balanced and antibalanced. Hence, our Theorem \ref{thm:pr} covers the conclusion of \cite[Theorem 4.4]{DPT21} and \cite[Theorem 1.2]{HW20}.
Next, we  use Theorem \ref{thm:pr} to derive the following results.
\begin{theorem}\label{thm:antibalanced_no}
     Let $\Gamma=(G,\sigma)$ be a connected antibalanced signed graph and $\{\lambda_{i}\}_{i=1}^n$ be the variational eigenvalues of $\Delta_p^{\sigma}$ with $p>1$. Then we have $\lambda_{n-1}<\lambda_{n}$ and there are no other eigenvalues between $\lambda_{n-1}$ and $\lambda_{n}$.
\end{theorem}
\begin{proof}
    Since $\Gamma$ is antibalanced, by Proposition \ref{pro:switching-spectra}, we can assume $\sigma\equiv-1$ without loss of generality. 
    
    We prove the theorem by contradiction. Assume that  $\lambda$ is an eigenvalue satisfying  $\lambda_{n-1}<\lambda<\lambda_n$ and $f$ is an  eigenfunction  corresponding to $\lambda$. By Theorem \ref{thm:nodal-anti-graph}, we get  $\overline{\mathfrak{S}}(f)\leq 1$. Then by definition of $\overline{\mathfrak{S}}$, we have $f\geq 0$ on every vertex or $f\leq 0$ on every vertex. We assume $f\geq 0$ on every vertex and the case that $f\leq 0$ on every vertex can be proved  similarly. If $f$ is zero on some $x\in V$, we have by the eigen-equation that  
\begin{equation*}
    \sum_{y\sim x}w_{xy}\Phi_p(f(x)+f(y))+\kappa_x\Phi_p(f(x))=\lambda\mu_{x}\Phi_p(f(x)).
\end{equation*}
So we have $  \sum_{y\sim x}w_{xy}\Phi_p(f(y))=0$.  Because $f(y)\geq 0$ for any $y\in V$, we obtain $f(y)=0$ for any $y\sim x$. By the connectedness of $\Gamma$, we have $f=0$ on all vertices, which can not happen. So $f$ is positive on all vertices. Then, we apply Theorem \ref{thm:pr} to get $\lambda=\lambda_n$, which leads to a contradiction. 
\end{proof}
 Using again the fact that a bipartite graph with $\sigma\equiv 1$ is antibalanced, we derive from Theorem \ref{thm:antibalanced_no} the following corollary. 
\begin{cor}\label{cor:5.1}
For any connected bipartite graph, there are no eigenvalues between the largest and the  second largest variational eigenvalues of the corresponding $p$-Laplacian with $p>1$.
\end{cor}

\section{Interlacing theorems}
\label{sec:interlacing}
 When one wants to understand a quantitative property of a graph, it is natural to investigate how this quantity changes under modifying the graph via deleting vertices or edges. 

 In this section, for an eigenpair $(\lambda,f)$ of $\Delta_p^\sigma$ with $p>1$, we give a way to modify a signed graph to a forest $T$ such that $(\lambda,f|_T)$ is again an eigenpair of $T$. We estimate how the eigenvalue changes in each step. This leads to a nonlinear version of the Cauchy Interlacing Theorem. The theorems in this section are signed versions of the theorems in \cite[Section 5]{DPT21}. Those interlacing theorems will be useful for the lower bound estimates of $\mathfrak{S}(f)$ in the next section.

 
\subsection*{Removing an edge}

Consider a signed graph $\Gamma=(G,\sigma)$, where $G=(V,E)$, with an edge measure $w$, a vertex weight $\mu$, and a potential function $\kappa$. Let $f\in C(V)$ be a function and $\{x_0,y_0\}\in E$ be an edge such that $f(x_0)f(y_0)\neq 0$. We define a new signed graph
\[\Gamma'=(G',\sigma')\,\, \text{where}\,\, G'=(V,E'),\,\,E':=E\setminus\{ x_0,y_0\},\,\,\text{and}\,\, \sigma_{xy}'=\sigma_{xy}\,\,\text{for any}\,\, \{x,y\} \in E',\]
with an edge measure $w'$, a vertex weight $\mu'$ and a potential function $\kappa'$ defined as follows: $w'_{xy}=w_{xy}$ for any $\{x,y\}\in E'$, $\mu_x'=\mu_x$ for any $x\in V$, and 
\begin{equation*}
\kappa_x'=\begin{cases}
		\kappa_x, &\text{if}\,\, x\neq x_0,y_0;  \\
		\kappa_x+w_{x_0y_0}\Phi_p\left(1-\sigma_{x_0y_0}\frac{f(y_0)}{f(x_0)}\right), & \text{if}\,\, x=x_0;\\
		\kappa_x+w_{x_0y_0}\Phi_p\left(1-\sigma_{x_0y_0}\frac{f(x_0)}{f(y_0)}\right), &\text{if}\,\, x=y_0.
	\end{cases}
\end{equation*}
Then, the corresponding $p$-Laplacian with $p>1$ of the new signed graph $\Gamma'$ is given by
\begin{equation}\label{eq:edge}
 \Delta_p^{\sigma'}g(x)=\sum_{y\in V:\{y,x\}\in E'}w_{xy}'\Phi_p\left(g(x)-\sigma'_{xy}g(y)\right)+\kappa'_x\Phi_p(g(x)).
\end{equation}

It is direct to check that the above choices of $w'$, $\mu'$ and $\kappa'$ lead to the following property: If $f\in C(V)$ is an eigenfunction corresponding to an eigenvalue $\lambda$ of the $p$-Laplacian $\Delta_p^\sigma$ with $p>1$, then $f$ is still an eigenfunction of $\Delta_p^{\sigma'}$ corresponding to $\lambda$. 

Let $\mathcal{R}_p^{\sigma'}$ be the Rayleigh quotient of $\Delta_p^{\sigma'}$  defined as 
\begin{equation*}
	\mathcal{R}_p^{\sigma'}(g)=\frac{\sum_{xy\in E'}w'_{xy}|g(x)-\sigma_{xy}'g(y)|^p+\sum_{x\in V}\kappa'_{x}|g(x)|^p}{\sum_{x\in V}\mu_{x}'|g(x)|^p}.
\end{equation*}

We recall the following lemma from \cite[Proposition 4.4]{Solimini}, which will be very useful in the proofs of Theorem \ref{Lemma:edge}, Lemma \ref{Lemma:node} and Theorem \ref{thm:node}.
\begin{lemma}\label{Lemma:test}
    Let $A$ be a centrally  symmetric subset in a Banach space with $\gamma(A)> k$. Let $\phi:A\to \mathbb{R}^k$ be a continuous odd map. Then we have $\gamma(\phi^{-1}(0))\geq \gamma(A)-k$.
\end{lemma}
\begin{theorem}\label{Lemma:edge} Consider a signed graph $\Gamma=(G,\sigma)$ where $G=(V,E)$ and a function $f\in C(V)$. Let $\Delta_p^{\sigma}$ be the corresponding $p$-Laplacian with $p>1$, and $\Gamma'=(G',\sigma')$, $\Delta_p^{\sigma'}$ be defined as above. Denote by $\lambda_k$ and $\eta_k$ the  $k$-th variational eigenvalues of $\Delta_p^{\sigma}$ and $\Delta_p^{\sigma'}$, respectively. Then we have:
\begin{itemize}
	\item [(i)]if $f(x_0)\sigma_{x_0y_0}f(y_0)<0$, then $\eta_{k-1}\leq \lambda_k\leq \eta_k$ for any $1<k\leq n$;
	\item [(ii)]if $f(x_0)\sigma_{x_0y_0}f(y_0)>0$, then  $\eta_k\leq \lambda_k \leq \eta_{k+1}$ for any $1\leq k< n$.
\end{itemize}
\end{theorem}
\begin{proof}
 We first assume $f(x_0)\sigma_{x_0y_0}f(y_0)<0.$ Setting $\mathcal{E}(g):=\mathcal{R}_p^{\sigma'}(g)-\mathcal{R}_p^{\sigma}(g)$ for any $g:V\to \mathbb{R}$, we compute
\begin{equation}
	\begin{aligned}
	\frac{\Vert g \Vert_p^p\mathcal{E}(g)}{w_{x_0y_0}}=-|g(x_0)-\sigma_{x_0y_0}g(y_0)|^p+\left(\frac{|g(x_0)|^p}{\Phi_p(f(x_0))}-\sigma_{x_0y_0}\frac{|g(y_0)|^p}{\Phi_p(f(y_0))}\right)\Phi_p(f(x_0)-\sigma_{x_0y_0}f(y_0)).
	\end{aligned}
\end{equation}
Applying Lemma \ref{lemma:elementary} by taking $a=f(x_0)$, $b=-\sigma_{x_0y_0}f(y_0)$, $ta=g(x_0)$, and $sb=-\sigma_{x_0y_0}g(y_0)$, we have $\mathcal{E}(g)\geq 0$, and hence
\[\mathcal{R}_p^{\sigma}(g)\leq \mathcal{R}_p^{\sigma'}(g)\,\,\text{for any}\,\, g: V\to \mathbb{R},\]
where the equality holds if and only if  $g(x_0)f(y_0)-g(y_0)f(x_0)=0$.

Let $A_k\in \mathcal{F}_k(\mathcal{S}_p)$ be a  set such that $	\lambda_k=\max_{g\in A_k}\mathcal{R}_p^{\sigma}(g).$ Define
   \begin{equation*}
   	\begin{aligned}
	\phi: A_k&\to \mathbb{R},\\
	g&\mapsto g(x_0)f(y_0)-g(y_0)f(x_0).
   	\end{aligned}
   \end{equation*} 
Observe that $\phi$ is odd. Since $k>1$ and $\gamma(A_k)\geq k$,  we use   Lemma \ref{Lemma:test} to get $\gamma(\phi^{-1}(0))\geq k-1$. Moreover, we have $\mathcal{R}_p^{\sigma}(g)=\mathcal{R}_p^{\sigma'}(g)$ for any $g\in \phi^{-1}(0)$. So we derive that
\begin{equation}
	\eta_{k-1}=\min_{A\in \mathcal{F}_{k-1}(\mathcal{S}_p)}\max_{g\in A}\mathcal{R}_p^{\sigma'}(g)\leq \max_{g\in \phi^{-1}(0)}\mathcal{R}_p^{\sigma'}(g)=\max_{g\in \phi^{-1}(0)}\mathcal{R}_p^{\sigma}(g)\leq\max_{g\in A_k}\mathcal{R}_p^{\sigma}(g)=\lambda_k.
\end{equation} 
By definition of variational eigenvalues, we  obtain
\begin{equation}
	\lambda_k=\min_{A\in \mathcal{F}_{k}(\mathcal{S}_p)}\max_{g\in A}\mathcal{R}_p^{\sigma}(g)\leq \min_{A\in \mathcal{F}_{k}(\mathcal{S}_p)}\max_{g\in A}\mathcal{R}_p^{\sigma'}(g)\leq \eta_k.
\end{equation}
This proves (i). The proof of (ii) follows similarly.
\end{proof}
\begin{remark}
	We  define $\lambda_{m}=\eta_{m}=-\infty$ for $m\leq 0$ and $\lambda_{m}=\eta_{m-1}=+\infty$ for $m>n$. Then the above theorem holds for any $k\in \mathbb{Z}$.
\end{remark}

\subsection*{Removing a node}  
Consider a signed graph $\Gamma=(G,\sigma)$ and the  corresponding $p$-Laplacian $\Delta_p^\sigma$ with $p>1$. For a given vertex $x_0\in V$, we define a new signed graph $\Gamma'=(G',\sigma')$ where $G'=(V',E')$ is the subgraph induced by $V':=V\setminus \{x_0\}$, and $\sigma'_{xy}=\sigma_{xy}$ for any $\{x,y\}\in E'$, with an edge weight $w'$, a vertex weight $\mu'$ and a potential function $\kappa'$ defined as follows: $w_{xy}'=w_{xy}$ for any $\{x,y\}\in E'$, $\mu_x'=\mu_x$ for any $x\in V'$, and $\kappa_x'=\kappa_x+w_{xx_0}$ for any $x\in V'$. 

Then, we define the corresponding $p$-Laplacian $\Delta_p^{\sigma'}$ on $\Gamma'$ as follows:
\begin{equation}\label{eq:node}
	\Delta_p^{\sigma'}g(x)=\sum_{y\in V':\{x,y\}\in E'}w_{xy}'\Phi_p(g(x)-\sigma'_{xy}g(y))+ \kappa'_x\Phi_p(g(x)),
\end{equation}

For convenience, we define two maps $\Psi:C(V)\to C(V')$ and $\psi:C(V')\to C(V)$ between the function spaces $C(V):=\{f: V\to \mathbb{R}\}$ and $C(V'):=\{f: V'\to \mathbb{R}\}$ as follows:
For any $g:V\to \mathbb{R}$, we define $(\Psi g)(x)=g(x)$ for any $x\in V'=V\setminus\{x_0\}$; For any $h:V'\to \mathbb{R}$, we define $(\psi h)(x)=h(x)$ for any $x\in V'=V\setminus\{x_0\}$ and $(\psi h)(x_0)=0$.

The reason to choose the new $w'$, $\mu'$ and $\kappa'$ as above is to ensure that, for an eigenfunction $f$ of $\Delta_p^\sigma$ corresponding to an  eigenvalue $\lambda$ such that $f(x_0)=0$, $\Psi f$ is an eigenfunction of $\Delta_p^{\sigma'}$ corresponding to the same eigenvalue $\lambda$. Indeed, we have for any $x\in V'$ that
\begin{equation*}
	\begin{aligned}
		\Delta_p^{\sigma'}(\Psi f)(x)&=\sum_{y\in V':\{x,y\}\in E'}w_{xy}'\Phi_p\left(\Psi f(x)-\sigma'_{xy}\Psi f(y)\right)+\kappa'_x\Phi_p\left(\Psi f(x)\right)\\
		&=\sum_{y\in V':\{x,y\}\in E'}w_{xy}\Phi_p(f(x)-\sigma'_{xy}f(y))+ (\kappa_x+w_{xx_0})\Phi_p(f(x))\\
		&=\sum_{y\in V':\{x,y\}\in E'}w_{xy}\Phi_p(f(x)-\sigma_{xy}f(y))+ w_{xx_0}\Phi_p(f(x)-\sigma_{xx_0}f(x_0))+ \kappa_x\Phi_p(f(x))\\
		&=\sum_{y\in V:\{x,y\}\in E}w_{xy}\Phi_p(f(x)-\sigma_{xy}f(y))+ \kappa_x\Phi_p(f(x))\\
		&=\lambda \mu_{x}\Phi_p(f(x))=\lambda \mu'_{x}\Phi_p(\Psi f(x)).
	\end{aligned}
\end{equation*}

Let  $\mathcal{R}_p^{\sigma'}$ be the Rayleigh quotient of $\Delta_p^{\sigma'}$  defined as 
\begin{equation*}
	\mathcal{R}_p^{\sigma'}(g)=\frac{\sum_{xy\in E'}w'_{xy}|g(x)-\sigma_{xy}'g(y)|^p+\sum_{x\in V'}\kappa'_{x}|g(x)|^p}{\sum_{x\in V'}\mu'_{x}|g(x)|^p}.
\end{equation*}

It is direct to check the following facts:
\begin{itemize}
\item For any $g\in C(V)$ with $g(x_0)=0$, we have $\mathcal{R}_p^{\sigma}(g)=\mathcal{R}_p^{\sigma'}(\Psi g)$;
\item For any $h\in C(V')$, we have $\mathcal{R}_p^{\sigma}(\psi h)=\mathcal{R}_p^{\sigma'}( h)$.
\end{itemize} 
\begin{lemma}\label{Lemma:node}
Consider a signed graph $\Gamma=(G,\sigma)$ where $G=(V,E)$ and a given vertex $x_0\in V$. 
Let $\Delta_p^{\sigma}$ be the corresponding $p$-Laplacian with $p>1$, and $\Gamma'=(G',\sigma)$, $\Delta_p^{\sigma'}$ be defined as above. Denote by $\lambda_k$ and $\eta_k$  the $k$-th variational eigenvalues of $\Delta_p^{\sigma}$ and $\Delta_p^{\sigma'}$, respectively. Then we have
\begin{equation*}
	\lambda_k\leq \eta_k\leq \lambda_{k+1}, \,\,\text{for any}\,\, 1\leq k\leq n-1. 
\end{equation*}
\end{lemma}
\begin{proof}
Define $\mathcal{S}_p'=\{g:V'\to \mathbb{R}\left|\,\sum_{x\in V'}\mu_x|g(x)|^p=1\right.\}$.  Let $A_k'\in \mathcal{F}_k(\mathcal{S}_p')$ be a set such that $\eta_k=\max_{g\in A_k'}\mathcal{R}_p^{\sigma'}(g)$. Define $A_k:=\psi(A_k')$. By definition, we have $A_k\in \mathcal{F}_k(\mathcal{S}_p)$, and
\begin{equation*}
	\lambda_k=\min_{A\in \mathcal{F}_{k}(\mathcal{S}_p)}\max_{g\in A}\mathcal{R}_p^{\sigma}(g)\leq \max_{g\in A_k}\mathcal{R}_p^{\sigma}(g)= \max_{g\in A_k'}\mathcal{R}_p^{\sigma'}(g)=\eta_k.
\end{equation*}
This concludes the proof of the first inequality.

Let $A_{k+1}\in \mathcal{F}_{k+1}(\mathcal{S}_p)$ be a set such that $
\lambda_{k+1}=\max_{g\in A_{k+1}}\mathcal{R}_p^{\sigma}(g)$. 
Define
\begin{equation*}
\begin{aligned}
	\phi:A_{k+1} &\to \mathbb{R} \\
      	 g&\mapsto g(x_0).
\end{aligned}
\end{equation*}
By Lemma \ref{Lemma:test}, we have $\phi^{-1}(0)\subset \mathcal{F}_k(\mathcal{S}_p)$ and $\Psi(\phi^{-1}(0))\subset \mathcal{F}_k(\mathcal{S}_p')$. So we  get
\begin{equation*}
   \eta_k= \min_{A\in \mathcal{F}_{k}(\mathcal{S}_p')}\max_{g\in A}\mathcal{R}_p^{\sigma'}(g)\leq \max_{g\in \Psi (\phi^{-1}(0))}\mathcal{R}_p^{\sigma'}(g)=\max_{g\in  \phi^{-1}(0)}\mathcal{R}_p^{\sigma}(g)\leq \max_{g\in A_{k+1}}\mathcal{R}_p^{\sigma}(g)=\lambda_{k+1}.
\end{equation*}
This concludes the proof of the second inequality.
\end{proof}
We can use Lemma \ref{Lemma:node} iteratively to get the following theorem.
\begin{theorem}\label{thm:node}
Consider a signed graph $\Gamma=(G,\sigma)$ where $G=(V,E)$, with an edge measure $w$, a vertex weight $\mu$ and a potential function $\kappa$. Given a subset $\{x_1,\ldots,x_m\}\subset V$ of $m$ vertices, we define a new signed graph $\Gamma'=(G',\sigma')$, where $G'=(V',E')$ is the subgraph induced by $V':=V\setminus\{x_1,\ldots,x_m\}$, with an edge measure $w'$, a vertex weight $\mu'$ and a potential function $\kappa'$ defined as follows: $w_{xy}'=w_{xy}$ for any $\{x,y\}\in E'$, $\mu_x'=\mu_x$ for any $x\in V'$, and \[\kappa_x'=\kappa_x+\sum_{i=1}^mw_{xx_i},\,\,\text{for any}\,\, x\in V'.\]   
Denote by $\{\lambda_i\}_{i=1}^n$ and $\{\eta_i\}_{i=1}^{n-m}$ the variational eigenvalues of the corresponding $p$-Laplacians $\Delta_p^{\sigma}$ and $\Delta_p^{\sigma'}$ with $p>1$, respectively. Then, we have  
\begin{equation*}
	\lambda_k\leq \eta_k\leq \lambda_{k+m},\,\,\text{for any}\,\, 1\leq k\leq n-m.
\end{equation*}
\end{theorem}

\section{Lower bounds of the number of strong nodal domains}
\label{sec:lower-nodal}
 In this section, we prove the lower bound estimates of the strong nodal domains. For convenience, we give the following symbols.

 Given a signed graph $\Gamma=(G,\sigma)$ where $G=(V,E)$, we denote by $c(G)$ the number of the connected components of $G$. For a given function $g:V\to \mathbb{R}$, we define two edge sets \[E_{g^+}=\{  \{x,y\}\in E:g(x)\sigma_{xy}g(y)>0  \} \,\,\text{and}\,\,E_{g^-}=\{  \{x,y\}\in E: g(x)\sigma_{xy}g(y)<0  \},\] and two signed graphs \begin{align*}
     \Gamma_{g^+}=(G_{g^+},\sigma)\,\,\text{with}\,\,G_{g^+}=(V,E_{g^+}) \,\,\text{and}\,\, \Gamma_{g^-}=(G_{g^-},\sigma)\,\,\text{with}\,\,G_{g^-}=(V,E_{g^-}). 
\end{align*}
      Let
 \begin{align*}
     l(G):=|E|-&|V|+c(G),\,\, l(\Gamma_{g^+}):=|E_{g^+}|-|V|+c(G_{g^+}) \\
 &\text{and}\,\, l(\Gamma_{g^-}):=|E_{g^-}|-|V|+c(G_{g^-}).
 \end{align*} 
 Notice that the number $l(G)$ is the minimal number of edges that need to be removed from $G$ in order to turn it into a forest. 
 We further denote by $z(g)$ the number of zeros of $g$.

The theorem below is our main result in this section.
\begin{theorem}\label{thm:13}
      Let $\Gamma=(G,\sigma)$ be a connected signed graph. Let $f$ be an eigenfunction of $\Delta_p^{\sigma}$ corresponding to an eigenvalue $\lambda$,  and   $\lambda_{1}\leq \ldots\leq \lambda_n$ be the variational eigenvalues of $\Delta_p^{\sigma}$,  where $p>1$. Assume that $\{x_i\}_{i=1}^{z(f)}$ are the zero vertices of $f$. Define $\Gamma'=(G',\sigma)$, where $G'=(V',E')$ is the subgraph induced by $V':=V\setminus\{x_i\}_{i=1}^{z(f)}$. Then we have:
\begin{itemize}
	\item [(i)]if $\lambda>\lambda_k$, then $\mathfrak{S}(f)\geq k-l(G')+l(\Gamma_{f^+})-z(f)+c(G')$;
	\item [(ii)]if $\lambda=\lambda_k>\lambda_{k-1}$ and the multiplicity of $\lambda_k$ is $r$, then $\mathfrak{S}(f)\geq k+r-1-l(G')+l(\Gamma_{f^+})-z(f)$.
\end{itemize} 
\end{theorem}
This theorem can be regarded as a signed version of \cite[Theorem 3.10]{DPT21} for generalized $p$-Laplacian on graphs, which is an extension of previous results on $2$-Laplacian in the work of Berkolaiko \cite{Berkolaiko08} and Xu-Yau \cite{XY12}. The lower bound estimates of strong nodal domains for $2$-Laplacians on signed graphs have been discussed in \cite{Ali16, GL21+}. Restricting to the linear case $p=2$, Theorem \ref{thm:13} is, in fact, weaker than \cite[Theorem 6.6]{GL21+}. The estimate in \cite[Theorem 6.6]{GL21+} uses the cardinality of the so-called Fiedler zero set -- a special subset of the whole zero set -- instead of $z(f)$. It is still open whether the lower bound in \cite[Theorem 6.6]{GL21+} can be extended to the current setting or not. 

We first prove the following lemma.
\begin{lemma}\label{Lemma:E-}
	Let $\Gamma=(G,\sigma)$ be a signed graph with $G=(V,E)$ and $(\lambda,f)$ be an eigenpair of $\Delta_p^{\sigma}$ on $\Gamma$. 
Denote  by $E_z$ the set of edges incident to the zero  vertices of $f$. Then we have
\begin{equation}\label{eq:-}
	|E_{f^-}|=|E|-|E_z|+z(f)-|V|-l(\Gamma_{f^+})+\mathfrak{S}(f)\leq |E|-|V|+\mathfrak{S}(f)-l(\Gamma_{f^+}).
\end{equation}
\end{lemma}
\begin{proof}

First, we remove  all  zero vertices of $f$ on $\Gamma=(G,\sigma)$ to get an induced subgraph $\Gamma'=(G',\sigma')$ with $G'=(V',E')$. 
 By definition, we have 
 	$|E'|=|E|-|E_z|$.
  
Next, assume $E_{f^+}=\{e_i\}_{i=1}^m$ with $m=|E_{f^+}|$ and $e_i=\{x_i,y_i\}$. We remove  all edges in $E_{f^+}$ one by one to get the graph $\Gamma''=(G'',\sigma'')$ with $G''=(V'',E'')$ at end. At the $i$-th step, we remove the edge $e_i$. 
For $j=1,\ldots,m$, let $\Gamma^j=(G^j,\sigma^{(j)})$ be the signed graph which is obtained by removing the edges $\{e_i\}_{i=1}^{j}$ from $\Gamma$.
We denote $\Delta l^+(e_j,f)=l(\Gamma^j_{f^+})-l(\Gamma^{j-1}_{f^+})$, and define $\Delta v(e_j,f)$ to be the variation between the number of nodal domains of $f$ on $\Gamma_j$ and $\Gamma_{j-1}$, where we use $\Gamma_0$ to denote $\Gamma$. By a direct computation, we have for any $j=1,\ldots,m$ that
\begin{equation}\label{eq:add}
\Delta v(e_j,f)-\Delta l^+(e_j,f)=
\left\{
\begin{aligned}
	  0, &\quad f(x_j)\sigma_{x_jy_j}f(y_j)<0; \\
	  1, &\quad  f(x_j)\sigma_{x_jy_j}f(y_j)>0.
\end{aligned}
\right.
\end{equation}
Therefore, we derive that
\begin{equation}\label{eq:7.1}
    \sum_{j=1}^{m}(\Delta v(e_j,f)-\Delta l^+(e_j,f))=|E_{f^+}|=|E'|-|E_{f^-}|=|E|-|E_z|-|E_{f^-}|.
\end{equation}
On the graph $\Gamma''$, there are no strong nodal domain walks of $f$. Hence, $l^+(\Gamma_{f^+}'')=0$ and the number of nodal domains of $f$ on $\Gamma''$ is $|V|-z(f)$. Then, we have \begin{equation}\label{eq:7.2}
\sum_{j=1}^{m}\Delta v(e_j,f)=|V|-z(f)-\mathfrak{S}(f)\,\,\text{and}\,\, \sum_{j=1}^{m}\Delta l^+(e_j,f)=-l(\Gamma_{f^+}).\end{equation}
Combining (\ref{eq:7.1}) and (\ref{eq:7.2}), we have  
\begin{equation*}
	|V|-z(f)-\mathfrak{S}(f)+l(\Gamma_{f^+})=|E|-|E_z|-|E_{f^-}|.
\end{equation*}
This implies 
\begin{equation*}
\begin{aligned}
	|E_{f^-}|&=|E|-|E_z|-|V|+z(f)-l(\Gamma_{f^+})+\mathfrak{S}(f)\\
	&\leq |E|-|V|-l(\Gamma_{f^+})+\mathfrak{S}(f).
\end{aligned}
\end{equation*}
The last inequality is because of $|E_z|\geq z(f)$. Then we complete the proof.
\end{proof}

\begin{proof}[Proof of Theorem \ref{thm:13} (i)]
	First, since $\Gamma'$ is obtained by removing all  zero vertices of $f$ from $\Gamma$, we can define a new $p$-Laplacian on $\Gamma'$ as \eqref{eq:node} denoted by $\Delta_p^{\sigma'}$. Next, we remove all the edges in  $E'_{f^-}$ of $f$ on $\Gamma'$ one by one to get the graph $\Gamma''=(G'',\sigma'')$ with $G''=(V'',E'')$ at end. At each step, we define a new $p$-Laplacian as in \eqref{eq:edge}. Denote by $\Delta_p^{\sigma''}$ the $p$-Laplacian on we obtain at end. By Theorem  \ref{Lemma:edge} and Theorem \ref{thm:node}, we  get
	\begin{equation*}
		\lambda>\lambda_k\geq \lambda_{k-z(f)}'\geq \lambda''_{k-z(f)-|E_{f^-}|}. 
	\end{equation*}
For any $\{x,y\}\in E''$, we have $f(x)\sigma_{xy}f(y)>0$. Define $\tau(x) =\frac{f(x)}{|f(x)|}$ for any $x\in V''$. It is a switching function such that $\sigma^{\tau}\equiv 1$ and $\tau f$ is positive on all vertices in $G''$. By switching invariance of  eigenvalues and the Perron-Frobenius type theorem \cite[Theorem 4.1]{DPT21},  $\lambda$ is the first variational eigenvalue of $\Delta_p^{\sigma''}$. Since $		\lambda> \lambda''_{k-z(f)-|E_{f^-}|}$, 
we have 
\begin{equation*}
	k-z(f)-|E_{f^-}|\leq 0.
\end{equation*}
We use Lemma \ref{Lemma:E-} to obtain
\begin{equation*}
\begin{aligned}
	\mathfrak{S}(f)&\geq k-z(f)-(|E|-|E_z|)+(|V|-z(f))+l(\Gamma_{f^+})\\
	&\geq k-z(f)-l(G')+c(G')+l(\Gamma_{f^+}).
\end{aligned}
\end{equation*}
This concludes the proof of (i).
\end{proof}
\begin{proof}[Proof of Theorem \ref{thm:13} (ii)]
As above, we first define a new $p$-Laplacian on $\Gamma'$ as in \eqref{eq:node}  denoted by $\Delta_p^{\sigma'}$. Let $\{\lambda_i'\}_{i=1}^{n-z(f)}$ be the variational eigenvalues of $\Delta_p^{\sigma'}$. By Theorem \ref{thm:node}, we have 
\[\lambda_{k+r-1-z(f)}'\leq \lambda \leq \lambda_{k+r-1}'.\]
Then there is a unique $h\in \mathbb{N}$ such that $\lambda\in [\lambda_{h}',\lambda_{h+1}')$. So we have $\lambda_{k+r-1-z(f)}'<\lambda'_{h+1}$. This implies $h+1>k+r-1-z(f) $.

Next, we remove $l(G')$ edges of $\Gamma'$ to make $\Gamma'$ to be a forest $T$. Assume that $\{e_i\}_{i=1}^{l(G')}$ are all the edges we remove, where $e_i=\{x_i,y_i\}$.  We define $\Gamma_j$ as the subgraph obtained by removing edges $\{e_i\}_{i=1}^j$ from $\Gamma'$. At each step, we define a new $p$-Laplacian on $\Gamma_j$ as in \eqref{eq:edge} denoted by $\Delta_{p,j}^{\sigma}$. Denote by $\{\lambda_k^{(j)}\}_{k=1}^{n-z(f)}$ the variational eigenvalues of $\Delta_{p,j}^{\sigma}$.

At the $j$-th  step, suppose that 
 $\lambda\in \left[\lambda_l^{(j)},\lambda^{(j)}_{l+1}\right)$. By Theorem \ref{Lemma:edge}, we have  $\lambda\in \left[\lambda^{(j+1)}_{l-1},\lambda^{(j+1)}_{l+1}\right)$ if $f(x_j)\sigma_{x_jy_j}f(y_j)<0$, and $\lambda\in \left[\lambda^{(j+1)}_{l},\lambda^{(j+1)}_{l+2}\right)$ if $f(x_j)\sigma_{x_jy_j}f(y_j)>0$.
  Define 
\begin{equation*}
	\Delta n(e_j,f)=\left\{
	\begin{aligned}
		    -1, \quad &\text{if\,\,} \lambda\in \left[\lambda^{(j+1)}_{l-1},\lambda^{(j+1)}_{l}\right); \\
		     0, \quad &\text{if\,\,} \lambda\in \left[\lambda^{(j+1)}_{l},\lambda^{(j+1)}_{l+1}\right);\\
			+1, \quad  &\text{if\,\,} \lambda\in \left[\lambda^{(j+1)}_{l+1},\lambda^{(j+1)}_{l+2}\right),
	\end{aligned}
	\right.
\end{equation*}


 and
\begin{equation*}
	\Delta M(e_j,f)=\left\{
	\begin{aligned}
		-1,\quad &\text{if\,\,} f(x_j)\sigma_{x_jy_j}f(y_j)<0 \text{\,\,and\,\,}\lambda\in \left[\lambda^{(j+1)}_{l-1},\lambda^{(j+1)}_{l}\right);\\
		0,\quad &\text{if\,\,} f(x_j)\sigma_{x_jy_j}f(y_j)<0 \text{\,\,and\,\,}\lambda\in \left[\lambda^{(j+1)}_{l},\lambda^{(j+1)}_{l+1}\right);\\
		-1, \quad &\text{if\,\,} f(x_j)\sigma_{x_jy_j}f(y_j)>0 \text{\,\,and\,\,}\lambda\in \left[\lambda^{(j+1)}_{l},\lambda^{(j+1)}_{l+1}\right);\\
		0, \quad &\text{if\,\,} f(x_j)\sigma_{x_jy_j}f(y_j)>0 \text{\,\,and\,\,}\lambda\in \left[\lambda^{(j+1)}_{l+1},\lambda^{(j+1)}_{l+2}\right).
\end{aligned}
\right.
\end{equation*}
Recalling \eqref{eq:add}, we derive by a direct computation that
\begin{equation}\label{0}
	\Delta n(e_j,f)-\Delta M(e_j,f)=\Delta v(e_j,f)-\Delta l^+(e_j,f),
\end{equation}
where $v(e_j,f)$ and $\Delta l^+(e_j,f)$ are defined as in the proof of Lemma \ref{Lemma:E-}.
  Since $f$ has no zeros on the forest $T$, we have by Theorem \ref{thm:forest} that $\lambda$ is a variational eigenvalue. Suppose $\lambda=\eta_m<\eta_{m+1}$, where $\{\eta_i\}_{i=1}^{n-z(f)}$ are variational eigenvalues of the $p$-Laplacian on $T$. Assume that $f$ has $\mathfrak{S}_T(f)$ nodal domains on $T$. By Theorem \ref{thm:forest} again, we have $\mathfrak{S}_T(f)= m$.
By definition, there holds that
\begin{align}
\sum_{j=1}^{l(G')}\Delta n(e_j,f)&=m-h, \label{eq:low1}\\
\sum_{j=1}^{l(G')}\Delta v(e_j,f)&=\mathfrak{S}_T(f)-\mathfrak{S}(f),\label{eq:low2}\\
\sum_{j=1}^{l(G')}\Delta l^+(e_j,f)&=-l(\Gamma_{f^+}),\label{eq:low3}\\
\sum_{j=1}^{l(G')}\Delta M(e_j,f)&\geq -l(G').\label{eq:low4}
\end{align}
We insert \eqref{eq:low1}, \eqref{eq:low2}, \eqref{eq:low3} and \eqref{eq:low4} into \eqref{0} to get 
\begin{equation*}
	\mathfrak{S}(f)= \mathfrak{S}_T(f)+l(\Gamma_{f^+})-m+h-l(G')= l(\Gamma_{f^+})+h-l(G')\geq k+r-1-z(f)+l(\Gamma_{f^+})-l(G').
\end{equation*}
This concludes  the proof.
\end{proof}

\section*{Acknowledgement}
We thank the anonymous referee for many comments and suggestions which greatly helped us improve the quality of the presentation of our paper.
CG and SL are supported by the National Key R and D Program of China 2020YFA0713100, the National Natural Science Foundation of China (No. 12031017), and Innovation Program for Quantum Science and Technology 2021ZD0302902. DZ is supported by the Fundamental Research Funds for the Central Universities (No. 7101303088).

\end{document}